\documentclass[12pt]{article}

\usepackage{amssymb}
\usepackage{graphics}

\newcommand{\qed}{$\;\;\;\Box$}
\newenvironment{proof}{\par\smallbreak{\sl\bf Proof.~}}
{\unskip\nobreak\hfill \qed \par\medbreak}

\newcounter{claim}
\renewcommand{\theclaim}{\arabic{claim}}
{\par\medskip\par}

{\qed\par\smallbreak}
\newcommand{\D}{{\cal D}}

\newcommand{\N}{{\mathbb N}}
\newcommand{\R}{{\mathbb R}}

\newcommand{\K}{{\mathbb K}}

\newcommand{\CC}{{\cal C}}
\newcommand{\KK}{{\cal K}}

\newcommand{\LL}{{\cal L}}

\newcommand{\beq}{\begin{equation}}
\newcommand{\ee}{\end{equation}}

\renewcommand{\d}{\partial}

\newtheorem{thm}{Theorem}[section]
\newtheorem{lem}[thm]{Lemma}

\newtheorem{defn}[thm]{Definition}

\newtheorem{rem}[thm]{Remark}

\newcommand{\al}{\alpha}
\newcommand{\be}{\beta}
\newcommand{\ga}{\gamma}

\newcommand{\eps}{\varepsilon}
\newcommand{\vphi}{\varphi}

\newcommand{\om}{\tau}

\newcommand{\reff}[1]{(\ref{#1})}      

\title{Time-Periodic  Second-Order Hyperbolic Equations: \\
Fredholmness, Regularity, and Smooth Dependence
} 

\newcounter{thesame}
\author{
I.~Kmit
 \ \ \ L.~Recke\\
{\small
Institute of Mathematics, Humboldt University of Berlin,}
\\
{\small Rudower Chaussee 25, D-12489 Berlin, Germany }
\\
{\small
and Institute for Applied Problems of Mechanics and Mathematics, }
\\
{\small
Ukrainian Academy of Sciences,  Naukova St.\ 3b, 79060 Lviv,
Ukraine 
}
\\
{\small   E-mail:
{\tt kmit@informatik.hu-berlin.de}}\\[5mm]
{\small
Institute of Mathematics, Humboldt University of Berlin,}\\
{\small 
Rudower Chaussee 25, D-12489 Berlin, Germany}\\
{\small   E-mail:
{\tt recke@mathematik.hu-berlin.de}}
}
 
\date{}

\begin{document}

\maketitle

\begin{abstract}
\noindent
The paper concerns the general linear one-dimensional second-order 
hyperbolic equation
$$
\partial^2_tu  - a^2(x,t)\partial^2_xu + a_1(x,t)\d_tu + a_2(x,t)\d_xu + a_3(x,t)u=f(x,t), \quad x\in(0,1)
$$
with periodic conditions in time and Robin boundary conditions in space. 
Under a non-resonance condition 
(formulated in terms of the coefficients $a$, $a_1$, and  $a_2$) ruling out 
the small divisors effect, we  prove the  Fredholm alternative.
Moreover, we show that the solutions have  higher regularity if the data have higher regularity and
if additional non-resonance  conditions are fulfilled. Finally, we 
state a result about smooth dependence 
on the data, where perturbations of the  coefficient $a$ lead to the known loss of smoothness while 
perturbations of the coefficients $a_1$, $a_2$, and $a_3$ do not.
\end{abstract}



\section{Introduction}\label{sec:intr}

\subsection{Problem setting and main results}\label{sec:setting}

We address the questions of Fredholm solvability, regularity of solutions and  smooth dependence
on the data
 for the general linear one-dimensional second-order
hyperbolic equation
\begin{equation}\label{1}
\partial^2_tw  - a^2(x,t)\partial^2_xw + a_1(x,t)\d_tw + a_2(x,t)\d_xw + a_3(x,t)w=f(x,t), \quad 
x\in(0,1)
\end{equation}
subjected to periodic conditions in time
\begin{equation}\label{2}
\begin{array}{rcl}
w(x,t)&=&w(x,t+T),\quad 
x\in(0,1),
\end{array}
\end{equation}
and Robin boundary conditions in space
\begin{equation}\label{3}
\begin{array}{rcl}
\d_xw(0,t)&=&r_0(t)w(0,t),\\
\d_xw(1,t)&=&r_1(t)w(1,t).
\end{array}
\end{equation}
Here $T>0$ is a fixed real number. 
The 
functions $a,a_1,a_2,a_3,f: [0,1]\times \R \to \R$ 
and $r_0,r_1: \R \to \R$ are supposed to be 
$T$-periodic with respect to $t$ and to satisfy 
\beq
\label{a}
a(x,t)>0 \;\mbox{ for all } x \in [0,1] \mbox{ and }  t \in \R
\ee
and $\int_0^Ta(0,t)r_0(t)\,dt\ne 0$ or $\int_0^Ta(1,t)r_1(t)\,dt\ne 0$.
Without loss of generality, throughout the paper we will assume that
\beq
\label{ar}
\int_0^Ta(0,t)r_0(t)\,dt\ne 0.
\ee

We will simply write $C_T^l$ for the Banach space 
of  $T$-periodic in $t$ and $l$-times continuously differentiable functions $u : [0,1]\times\R\to\R$,
with the usual norm
\beq\label{norm}
\|u\|_\infty+\sum_{i=1}^l\left(\|\d_x^iu\|_\infty+\|\d_t^iu\|_\infty\right),
\ee
where
\beq\label{norm0}
\|u\|_\infty=\max_{0 \le x \le 1} \;\max_{t \in \R}|u(x,t)|.
\ee
Moreover, we let $\CC^l=C_T^l\times C_T^l$. The norm in $\CC^l$ is again given 
by \reff{norm}--\reff{norm0}  but $|\cdot|$ in \reff{norm0} is now used to denote the Euclidean norm in $\R^2$.
Also, by $C_T^l\left(\R\right)$ we will denote the Banach space of $T$-periodic  and
 $l$-times continuously differentiable functions $u : \R\to\R$. Similarly, let $C_T^\infty$ (resp., 
$C_T^\infty(\R)$)
 denote
the space of  $T$-periodic in $t$ and  infinity differentiable functions  $u : [0,1]\times\R\to\R$
 (resp.,  $u : \R\to\R$).

The problem \reff{1}--\reff{3} can be written as a problem for  a first-order
hyperbolic integro-differential system. Indeed, set $u=(u_1,u_2)$ and 
\beq\label{const_C}
C=\int_0^Ta(0,t)r_0(t)\,dt
\ee
and introduce linear bounded operators $N: \CC\mapsto \R$,  
$I,G: \CC\mapsto C_T(\R)$ and 
$J,F: \CC\mapsto \CC$ by
\begin{eqnarray}
\label{I}
[Iu](t)&=& \int_0^t\frac{u_1(0,\tau)+u_2(0,\tau)}{2}\,d\tau,\\
\label{J}
[Ju](x,t)&=& \int_0^x\frac{u_1(\xi,t)-u_2(\xi,t)}{2a(\xi,t)}\,d\xi,\\
\label{N}
Nu&=& \frac{1}{C}\int_0^T\left(\frac{u_1(0,t)-u_2(0,t)}{2}-a(0,t)r_0(t)[Iu](t)\right)\,dt,\\
\label{G}
[Gu](t)&=& [Iu](t)+Nu,\\
\label{F}
[Fu](x,t)&=& [Gu](t)+[Ju](x,t).
\end{eqnarray}
Moreover, we introduce the following notation:
\beq\label{bij}
\begin{array}{cc}
\displaystyle
b_{11}=\frac{a_1}{2}+\frac{a_2}{2a}+\frac{a\d_xa-\d_ta}{2a},\quad 
b_{12}=\frac{a_1}{2}-\frac{a_2}{2a}+\frac{a\d_xa-\d_ta}{2a},\\\displaystyle
b_{21}=\frac{a_1}{2}+\frac{a_2}{2a}+\frac{a\d_xa+\d_ta}{2a},\quad
b_{22}=\frac{a_1}{2}-\frac{a_2}{2a}-\frac{a\d_xa+\d_ta}{2a}.
\end{array}
\ee
In the new unknowns 
\beq\label{u_sol}
u_1=\partial_tw  + a(x,t)\partial_xw,\quad u_2=\partial_tw  - a(x,t)\partial_xw
\ee
the problem \reff{1}--\reff{3} reads as follows:
\beq\label{1s}
\begin{array}{ll}
\displaystyle
\partial_tu_1  - a(x,t)\partial_xu_1 + b_{11}(x,t)u_1+ b_{12}(x,t)u_2 = f(x,t)-[a_3Fu](x,t)  \\\displaystyle
\partial_tu_2  + a(x,t)\partial_xu_2 + b_{21}(x,t)u_1+ b_{22}(x,t)u_2 = f(x,t)-[a_3Fu](x,t), 
\end{array}
\ee
\beq\label{2s}
u_j(x,t) = u_j(x,t+T),\quad   j=1,2,
\ee
\beq\label{3s}
\begin{array}{rcl}
\displaystyle
u_1(0,t) &=&\displaystyle u_2(0,t)+2a(0,t)r_0(t)[Gu](t),\\
\displaystyle
u_2(1,t) &=&\displaystyle u_1(1,t)-2a(1,t)r_1(t)[Fu](1,t).
\end{array}
\ee
It is not difficult to check (see Section~\ref{sec:equival}) that the problems  \reff{1}--\reff{3} and 
 \reff{1s}--\reff{3s}
are equivalent in the sense of the classical solvability, namely, 
that any classical solution to  \reff{1}--\reff{3}
produces a  classical solution to \reff{1s}--\reff{3s} by means of the formula 
\reff{u_sol} and, vice versa, 
 any  classical solution to  \reff{1s}--\reff{3s} produces a classical solution to 
 \reff{1}--\reff{3} by means of the formula
\beq\label{w_sol}
w(x,t)=[Iu](t)+[Ju](x,t)+Nu.
\ee

We will work with the concepts of a  weak (continuously differentiable) solution to
\reff{1}--\reff{3} and  a  weak (continuous) solution to
 \reff{1s}--\reff{3s}, based on the integration along characteristics.
In order to switch to the weak formulations, 
let us introduce characteristics of the system \reff{1s}. 
Given  $j=1,2$, $x \in [0,1]$, and $t \in \R$, the $j$-th characteristic  is defined as 
the solution 
$\xi\in [0,1] \mapsto \om_j(\xi,x,t)\in \R$  of the initial value problem
\beq\label{char}
\partial_\xi\om_j(\xi,x,t)=\frac{(-1)^j}{a(\xi,\om_j(\xi,x,t))},\;\;
\om_j(x,x,t)=t.
\ee
In what follows we will write
\begin{eqnarray}
\label{cdef}
c_j(\xi,x,t)&=&\exp \int_x^\xi(-1)^j
\left(\frac{b_{jj}}{a}\right)(\eta,\om_j(\eta,x,t))\,d\eta,\\
\label{ddef}
d_j(\xi,x,t)&=&\frac{(-1)^jc_j(\xi,x,t)}{a(\xi,\om_j(\xi,x,t))}.
\end{eqnarray}
Due to the method of characteristics, a $C^1$-map $u:[0,1]\times \R \to \R^2$ is a solution to 
the  problem \reff{1s}--\reff{3s} if and only if
it satisfies the following system of integral equations
\begin{eqnarray}
\label{rep1}
\lefteqn{
u_1(x,t)=c_1(0,x,t)\Big[u_2(0,\om_1(0,x,t))}\nonumber\\
&&
+2a(0,\om_1(0,x,t))r_0(\om_1(0,x,t))[Gu](\om_1(0,x,t))\Big]
\nonumber\\
&&-\int_0^x d_1(\xi,x,t) b_{12}(\xi,\om_1(\xi,x,t))u_2(\xi,\om_1(\xi,x,t))d\xi\nonumber\\ 
&&+\int_0^x d_1(\xi,x,t)\Bigl(f(\xi,\om_1(\xi,x,t))-[a_3Fu](\xi,\om_1(\xi,x,t))\Bigr)\,d\xi,
\end{eqnarray}
\begin{eqnarray}
\label{rep2}
\lefteqn{
u_2(x,t)=c_2(1,x,t)\Big[
 u_1(1,\om_2(1,x,t))}\nonumber\\
&&
-2a(1,\om_2(1,x,t))r_1(\om_2(1,x,t))[Fu](1,\om_2(1,x,t))
\Big]
\nonumber\\
&&-\int_1^x d_2(\xi,x,t) b_{21}(\xi,\om_2(\xi,x,t))u_1(\xi,\om_2(\xi,x,t))\,d\xi\nonumber\\ 
&&+\int_1^x d_2(\xi,x,t)\Bigl(f(\xi,\om_2(\xi,x,t))-[a_3Fu](\xi,\om_2(\xi,x,t))\Bigr)\,d\xi.
\end{eqnarray}
As it follows from Section~\ref{sec:equival},  
if $(u_1,u_2)$ is a continuous vector-function satisfying the system 
\reff{rep1}--\reff{rep2}, then
 the function $w$  given by  \reff{w_sol} is continuously differentiable.
Hence, the notion of a weak solution to the problem   \reff{1}--\reff{3} 
can be naturally defined as follows:
\begin{defn}\label{defn:cont}
(i) A function $u\in\CC$ is called a weak solution to  \reff{1s}--\reff{3s} 
if it  satisfies \reff{rep1} and \reff{rep2}.

(ii) Let $u$ be a weak solution to \reff{1s}--\reff{3s}. Then the continuously differentiable 
function $w$ given by the formula  \reff{w_sol} is called a weak solution to 
 \reff{1}--\reff{3}.
\end{defn}

 Denote by  
$\KK_w$ the vector space of all  weak solutions to \reff{1}--\reff{3} with $f=0$. 
We are prepared to state the Fredholm alternative theorem.

\begin{thm}\label{thm:Fredh} Suppose 
\reff{a} and \reff{ar}. Moreover, assume that
\begin{eqnarray}\label{smooth}
a\in C_T^2,\quad a_1,a_2\in C_T^1,\quad
a_3\in C_T,\quad r_0,r_1\in C_T^1\left(\R\right)
\end{eqnarray}
and either
\beq
\label{small}
 \int_0^1\left[
\left(\frac{b_{11}}{a}\right)(\eta,\om_1(\eta,1,t))+
\left(\frac{b_{22}}{a}\right)(\eta,\om_2(\eta,0,\om_1(0,1,t)))\right]\,d\eta
\ne 0 
\ee
or
\beq
\label{small+}
\int_0^1\left[
\left(\frac{b_{22}}{a}\right)(\eta,\om_2(\eta,0,t))+
\left(\frac{b_{11}}{a}\right)(\eta,\om_1(\eta,1,\om_2(0,1,t)))\right]\,d\eta
\ne 0 
\ee
for all $t$.
Then the following is true:

(i) $\dim \KK_w <\infty$.

(ii) The  space  of all  $f \in  C_T$ such that there exists a weak solution 
to~(\ref{1})--(\ref{3}) is a closed subspace of codimension $\dim \KK_w$ in $C_T$.

(iii)  Either  $\dim \KK_w>0$ or for any $f \in C_T$ 
there exists exactly one weak solution $w$
to~(\ref{1})--(\ref{3}). 
\end{thm}

\begin{rem}\rm
It follows from non-resonance conditions \reff{small} and \reff{small+} that, 
in general, resonances are defined by coefficients $a$, $a_1$, and $a_2$ of the 
 second- and the full first-order part of the equation (\ref{1}). 
In the particular  case $\d_ta\equiv 0$ conditions \reff{small} and
 \reff{small+} can be written in the form
\beq
 \int_0^1
\frac{b_{11}(\eta)+
b_{22}(\eta)}{a(\eta)}\,d\eta=
\int_0^1
\frac{a(\eta)a_1(\eta)+a(\eta)a^\prime(\eta)}{a^2(\eta)}\,d\eta
\ne 0.\nonumber
\ee
This means that in this case the resonances do not depend on the coefficient $a_2$ 
 any more.
\end{rem}

To formulate a regularity result, we introduce the notation
\begin{eqnarray}
c_j^l(\xi,x,t)&=&\exp \int_x^\xi(-1)^j
\left(\frac{b_{jj}}{a}-l\frac{\d_ta}{a^2}\right)(\eta,\om_j(\eta,x,t))\,d\eta.\nonumber
\end{eqnarray}
Notice that $c_j^0(\xi,x,t)=c_j(\xi,x,t)$.

\begin{thm}\label{thm:reg}
Suppose \reff{a} and \reff{ar}.

 (i)  Given $k\ge 1$, assume that
\beq\label{k-smooth}
\begin{array}{ll}
a\in C_T^{k+1},\quad a_1,a_2,a_3,f\in C_T^{k},\quad
 r_0,r_1\in C_T^{k}\left(\R\right)
\end{array}
\ee
and one of the following conditions is fulfilled:
\begin{eqnarray}\label{small1}
c_1^l(0,1,t)c_2^l(1,0,\om_1(0,1,t))< 1 \mbox{ for all }  
t\in\R  \mbox{ and } 
l=0,1,\dots,k,
\end{eqnarray}
\begin{eqnarray}\label{small11}
c_1^l(0,1,t)c_2^l(1,0,\om_1(0,1,t))> 1 \mbox{ for all } 
t\in\R  \mbox{ and } 
l=0,1,\dots,k,
\end{eqnarray}
\begin{eqnarray}\label{small111}
c_2^l(1,0,t)c_1^l(0,1,\om_2(1,0,t))< 1 \mbox{ for all } 
t\in\R  \mbox{ and } 
l=0,1,\dots,k,
\end{eqnarray}
and 
\begin{eqnarray}\label{small1111}
c_2^l(1,0,t)c_1^l(0,1,\om_2(1,0,t))> 1 \mbox{ for all } 
t\in\R  \mbox{ and } 
l=0,1,\dots,k.
\end{eqnarray}
Then any weak solution to (\ref{1})--(\ref{3}) 
belongs to $C_T^{k+1}$.

(ii) Assume that $a$ is independent of $t$.
Moreover, let $a, a_1, a_2, a_3, f\in C_T^\infty$
and $r_0, r_1\in C_T^\infty(\R)$. If one of the conditions \reff{small}
and  \reff{small+}
is fulfilled, then any weak solution to   (\ref{1})--(\ref{3})
belongs to $C_T^\infty$.
\end{thm}

We finish this section with the theorem describing  smooth dependence of the
solutions on the data. With this aim, by means of $\eps\in[0,1)$, we perform small 
perturbations 
\beq
\begin{array}{cc}
a^\eps(x,t)=a(x,t,\eps), \quad a_1^\eps(x,t)=a_1(x,t,\eps), \quad
 a_2^\eps(x,t)=a_2(x,t,\eps), \nonumber\\
 a_3^\eps(x,t)=a_3(x,t,\eps),\quad f^\eps(x,t)=f(x,t,\eps),\quad r_0^\eps(t)=r_0(t,\eps), \quad r_1^\eps(t)=r_1(t,\eps)\nonumber
\end{array}
\ee
of the coefficients $a(x,t)$, $a_1(x,t)$, $a_2(x,t)$, $a_3(x,t)$, $f(x,t)$,
$r_0(t)$, and $r_1(t)$,  respectively. Below we will also keep the notation
\beq\label{coef_ohne_eps}
\begin{array}{cc}
a(x,t)=a(x,t,0),\quad a_1(x,t)=a_1(x,t,0),\quad a_2(x,t)=a_2(x,t,0),
\nonumber\\
a_3(x,t)=a_3(x,t,0),\quad f(x,t)=f(x,t,0),\quad r_0(t)=r_0(t,0),\quad r_1(t)=r_1(t,0)\nonumber
\end{array}
\ee
for the non-perturbed coefficients. 
Similar notation will be used for the solutions 
\beq\label{sol_eps}
w^\eps(x,t)=w(x,t,\eps), \,\,\, 
u^\eps(x,t)=u(x,t,\eps)  \nonumber  
\ee
of the corresponding perturbed problems.

We are prepared to write down a perturbed problem to \reff{1}--\reff{3}:
\begin{equation}\label{1eps}
\partial^2_tw^\eps  - a^{\eps}(x,t)^2\partial^2_xw^\eps + a_1^\eps(x,t)\d_tw^\eps + 
a_2^\eps(x,t)\d_xw^\eps + a_3^\eps(x,t)w^\eps=f^\eps(x,t), 
\end{equation}
\begin{equation}\label{2eps}
\begin{array}{rcl}
w^\eps(x,t)&=&w^\eps(x,t+T),\\
\d_tw^\eps(x,t)&=&\d_tw^\eps(x,t+T),
\end{array}
\end{equation}
\begin{equation}\label{3eps}
\begin{array}{rcl}
\d_xw^\eps(0,t)&=&r_0^\eps(t)w^\eps(0,t),\\
\d_xw^\eps(1,t)&=&r_1^\eps(t)w^\eps(1,t).
\end{array}
\end{equation}
and the corresponding perturbed problem to \reff{1s}--\reff{3s}:
\beq\label{1s_eps}
\begin{array}{ll}
\partial_tu_1^\eps  - a^\eps(x,t)\partial_xu_1^\eps + b_{11}^\eps(x,t)u_1^\eps+ b_{12}^\eps(x,t)u_2^\eps = 
f^\eps(x,t)-[a_3^\eps F^\eps u^\eps](x,t), \\
\partial_tu_2^\eps  + a^\eps(x,t)\partial_xu_2^\eps + b_{21}^\eps(x,t)u_1^\eps+ b_{22}^\eps(x,t)u_2^\eps = 
f^\eps(x,t)-[a_3^\eps F^\eps u^\eps](x,t), 
\end{array}
\ee
\beq\label{2s_eps}
u_j^\eps(x,t) = u_j^\eps(x,t+T),\; j=1,2,
\ee
\beq\label{3s_eps}
\begin{array}{rcl}
\displaystyle
u_1^\eps(0,t) &=&\displaystyle u_2^\eps(0,t)+2a^\eps(0,t)r_0^\eps(t)[G^\eps u^\eps](t),\\
\displaystyle
u_2^\eps(1,t) &=&\displaystyle u_1^\eps(1,t)-2a^\eps(1,t)r_1^\eps(t)[F^\eps u^\eps](1,t),
\end{array}
\ee
where the functions $b_{ij}^\eps$ and the operators $F^\eps$ and $G^\eps$ are given by
\reff{bij} and \reff{I}--\reff{F} with $a$, $a_1$, $a_2$, $a_3$,
 $r_0$, and $r_1$ replaced by 
$a^\eps$, $a_1^\eps$, $a_2^\eps$, $a_3^\eps$,
$r_0^\eps$,  and $r_1^\eps$, respectively.

\begin{thm}\label{thm:dep} Assume \reff{a} and \reff{ar}.  
Let $\dim  \KK_w=0$. 

{\bf(i)} Given a non-negative integer $k$, suppose
\beq\label{k-smooth_dep}
\begin{array}{ll}
a^\eps\in C^{k+1}\left([0,1];C_T^{k+2}\right),\quad a_1^\eps,a_2^\eps, a_3^\eps,f^\eps\in 
 C^{k+1}\left([0,1];C_T^{k+1}\right), 
\\ r_0^\eps,r_1^\eps\in  C^{k+1}\left([0,1];C_T^k\left(\R\right)\right) 
\end{array}
\ee
and assume that one of the conditions
 \reff{small1},
 \reff{small11},  \reff{small111}, and  \reff{small1111} is fulfilled. 
Then there exists  $\eps_0\le 1$ such that for all  $\eps\le\eps_0$
 there exists a unique weak solution $w^\eps$
to (\ref{1eps})--(\ref{3eps}). Moreover, it holds  $w^\eps\in C_T^{k+1}$, and the map
$
\eps\in[0,\eps_0)\mapsto w^\eps\in C^{k-\ga}
$
is $C^\ga$-smooth for any non-negative integer $\ga\le k$.

{\bf(ii)} Assume that $a^\eps$ is $t$-independent and 
$a^\eps, a_1^\eps, a_2^\eps, a_3^\eps, r_0^\eps, r_1^\eps, f^\eps$  are $C^\infty$-smooth.  Suppose  
one of the conditions \reff{small} and \reff{small+}.
Then there is $\eps_0>0$ such that for all  $\eps\le\eps_0$ 
 there exists a unique  weak solution $w^\eps$
to  (\ref{1eps})--(\ref{3eps}). Moreover, for all $k\in\N$ it holds $w^\eps\in C_T^{k}$,
and the map
$
\eps\in[0,\eps_0)\mapsto w^\eps\in 
C_T^k
$
is $C^\infty$-smooth.
\end{thm}

\begin{rem}\label{rem:reg}\rm
Theorem \ref{thm:reg} claims that, under a number of conditions ruling out 
resonances, more regular data ensure more regular solutions.
This entails, in particular, that under the conditions of Theorem \ref{thm:dep}
the kernel of the operator of the problem (\ref{1})--(\ref{3}) in $C_T^{k+1}$
 equals  $\KK_w$.
This makes the assumption  $\dim  \KK_w=0$   of Theorem \ref{thm:dep} rather natural.
\end{rem}

\begin{rem}\label{rem:reg1}\rm
For the sake of brevity in Theorem \ref{thm:dep} we did not consider the special case when the coefficient $a$ is $\eps$-independent.
In this case there is no loss of smoothness, i.e. the dependence on $\eps$ of the partial derivatives of the solution $w^\eps$ is as smooth as 
the dependence on $\eps$ of $w^\eps$ itself. Furthermore,
 the smooth dependence of $w^\eps$ and its partial derivatives on $\eps$ can be easily obtained by
applying the classical Implicit Function Theorem. 
Specifically, if \reff{k-smooth_dep} is satisfied and if  $a$ is $\eps$-independent, then the map
$\eps\in[0,\eps_0)\mapsto w^\eps\in \CC^{k+1}$
is $C^{k+1}$-smooth.
\end{rem}

We hope that Theorems \ref{thm:Fredh},    \ref{thm:reg}, and  \ref{thm:dep} will make possible
developing   a theory of 
local smooth continuation and  bifurcation
of time-periodic solutions to general semilinear boundary value second-order hyperbolic  problems 
of a single space variable.
Another interesting direction could be a  non-smooth analysis of the 
discussed problems
in the framework of  algebras of generalized functions (see, e.g., \cite{ober}).

\subsection{Related work}\label{sec:related}

The paper \cite{He} addresses time-periodic solutions to the homogeneous 
wave equation $\partial_t^2w=\partial_x^2w$ for $x \in [0,1]$ with inhomogeneous boundary conditions
$\beta \partial_tw(0,t)-\gamma  \partial_xw(0,t)=f_0(t)$ and 
$\delta \partial_tw(1,t)+\gamma  \partial_xw(1,t)=f_1(t)$, where the 
right-hand sides $f_0$ and $f_1$ are $C^1$-smooth and periodic.
It is shown that the solution  $C^1$-smoothly depends  on the coefficients $\beta, \gamma$, and $\delta$ 
with respect to the $L^2$-norm (see also \cite{Grimmer} for a similar result).
Remark that the smooth dependence result for a linear problem in general  does not imply
such a result for the  corresponding
semilinear problem because the superposition operator generated by a $C^1$-smooth function  is  
$C^1$-smooth from $L^2$ into  $L^2$ if and only if it is affine.

In the papers \cite{Kim,Kim1,Kim2,kmit_tel,Rudakov,Wang} the  Fredholm or
 isomorphism properties of the linear
 telegraph equation with 
constant coefficients are  used  to get solvability results for the 
corresponding semilinear problems.

In \cite{KR2,KR3} we investigated  time-periodic problems for the general 
(with coefficients depending on space and time)
first-order hyperbolic systems with reflection boundary conditions. 
We suggested a functional-analytic approach that allowed us
to prove the Fredholm alternative 
in the scale of Sobolev-type spaces of periodic functions (in the autonomous case  \cite{KR2})
as well as in the space of continuous functions (in the non-autonomous case  \cite{KR3}). 
In the former  case \cite{KR2} we  applied the  Fourier series expansion, as in \cite{Vejvoda}.
In the latter case, like to the present paper,
we  used a weak formulation based on integration along characteristic curves.
The  Fredholm solvability result was essentially used in the autonomous case
to prove a smooth dependence on parameters and on the data. The general non-autonomous situation is much 
more complicated (and we address it here). The reason is that higher solution regularity, that 
is strongly related to the  smooth dependence \cite{KR2}, can be
achieved only if additional non-resonance conditions are fulfilled
\cite[Section 2.3]{kmit}.
The main difference  between  the problem  \reff{1s}--\reff{3s} and the problem that was investigated 
in  \cite{KR2,KR3} is this: now a number of  integral terms contribute into the system \reff{1s} as well as 
into the boundary conditions 
\reff{3s}. To handle these terms,  we will use the smoothing property proved in \cite{kmit}.

In \cite{hopf} we applied our results from  \cite{KR3} to prove a Hopf bifurcation theorem for semilinear
hyperbolic systems.

\section{Equivalence of the problems (\ref{1})--(\ref{3}) and 
(\ref{1s})--(\ref{3s})}\label{sec:equival}

Here we prove that  the problems (\ref{1})--(\ref{3}) and 
(\ref{1s})--(\ref{3s}) are equivalent in the sense of the classical solvability
as well as in the sense of the weak solvability.

First show that if $w$ is a classical
solution to (\ref{1})--(\ref{3}), then $u=(u_1,u_2)$ given by \reff{u_sol}
is a classical solution to  (\ref{1s})--(\ref{3s}). With this aim 
we outline the derivation of \reff{w_sol}.
We will use the
equalities
\beq
\partial_tw = \frac{u_1+u_2}{2},\quad\partial_xw=\frac{u_1-u_2}{2a} \nonumber
\ee
resulting from \reff{u_sol}.  Integrating the second one in $x$ and then the first one in $t$, one gets
\beq\label{between}
w(x,t)=[Iu](t)+[Ju](x,t)+w(0,0).
\ee
In order to show that
\beq\label{w00}
w(0,0)=Nu, 
\ee
 we first integrate the first equality of \reff{u_sol} in $t$ over $[0,T]$, put there $x=0$, 
and use the time-periodicity and the first equation from the boundary conditions \reff{2}. Consequently, we have
\beq
\int_0^T\Bigl(u_1(0,t)-a(0,t)r_0(t)w(0,t)\Bigr)\,dt=0. \nonumber
\ee
Then, calculating  $w(0,t)$ by means of  \reff{between}, the last
equality can be expressed in the form
\begin{eqnarray*}
w(0,0)\int_0^Ta(0,t)r_0(t)\,dt
=\int_0^T\frac{u_1(0,t)-u_2(0,t)}{2}\,dt-\int_0^Ta(0,t)r_0(t)
[Iu](t)\,\,dt,
\end{eqnarray*}
which in the notation of \reff{const_C} and \reff{N} gives \reff{w00} as desired.

Now, on the account of \reff{u_sol} and \reff{w_sol}, we easily come from  (\ref{1})--(\ref{3}) to 
(\ref{1s})--(\ref{3s}) where the latter is satisfied in the classical sense.

Further our aim is to prove that $w$ given by  \reff{w_sol} belongs to $C_T^2$ whenever $u\in\CC^1$ is a classical
solution to (\ref{1s})--(\ref{3s}). 
It suffices to show that $Ju\in C_T^2$ for every such
$u$.  By definition \reff{J}, we are done if we show that 
\beq\label{eqv1}
\int_0^x\frac{u_j(\xi,t)}{a(\xi,t)}\,d\xi\in C_T^2
\ee
 for $j=1,2$. Let us do this for $j=1$
(for $j=2$ we apply a similar argument).
Fix an arbitrary  $u\in\CC^1$ satisfying (\ref{1s})--(\ref{3s}).
 Plugging the  representation \reff{rep1} for $u_1$
into the integral  \reff{eqv1}, we see that we have to treat integral operators
of two kinds, namely
\beq\label{eqv2}
[S_1u_1](x,t)=\int_0^x\frac{c_1(0,\xi,t)u_2(0,\om_1(0,\xi,t))}{a(\xi,t)}\,d\xi
\ee
and 
\beq\label{eqv3}
[S_2u_1](x,t)=\int_0^x\frac{1}{a(\xi,t)}\int_0^\xi d_1(0,\eta,t)b_{12}(\eta,\om_1(\eta,\xi,t))
u_2(\eta,\om_1(\eta,\xi,t))\, d\eta d\xi
\ee
showing that they are smoothing and map $C_T^1$ into $C_T^2$. 

Denote by $\tau\in\R \mapsto \tilde\om_i(\tau,x,t)\in[0,1]$  the inverse of
the  equation of the $i$-th characteristic curve of \reff{1s}
passing through the point $(x,t)\in[0,1]\times\R$. 
Moreover,  in the calculations below we will use the formulas:
\begin{eqnarray}
\label{tilde_dx}
\d_t\tilde\om_j(\tau,x,t) & = &(-1)^{j+1}a(x,t) \exp \int_t^\tau(-1)^j
\d_1a(\tilde\om_j(\eta,x,t),\eta) d \eta,
\\
\label{tilde_dt}
\d_x\tilde\om_j(\tau,x,t) & = & \exp \int_t^\tau(-1)^j
\d_1a(\tilde\om_j(\eta,x,t),\eta) d \eta.
\end{eqnarray}

By simple change of variables in \reff{eqv2},
we get the following representation for
$[S_1u_1](x,t)$:
$$
[S_1u_1](x,t)=\int_t^{\om_1(0,x,t)}
\frac{c_1(0,\tilde\om_1(t,0,\tau),t)\d_\tau\tilde\om_1(t,0,\tau)u_2(0,\tau)}{a(\tilde\om_1(t,0,\tau),t)}\,d\tau.
$$
Taking into account \reff{tilde_dx} and smoothness assumptions on the initial data, 
we conclude that the right-hand side is a $C^2$-function as desired.

It remains to treat \reff{eqv3}. To this end, let
$$
d_{12}(\eta,\xi,t)=a(\xi,t)^{-1} d_1(0,\eta,t)b_{12}(\eta,\om_1(\eta,\xi,t)).
$$
By Fubini's theorem, 
\beq\label{fubini}
[S_2u_1](x,t)=\int_0^x\int^x_\eta d_{12}(\eta,\xi,t)
u_2(\eta,\om_1(\eta,\xi,t))\, d\xi d\eta.
\ee
Hence,
\begin{eqnarray}
\lefteqn{
\d_t[S_2u_1](x,t)=\int_0^x\int^x_\xi \d_t d_{12}(\eta,\xi,t)
u_2(\eta,\om_1(\eta,\xi,t))\, d\eta d\xi}\nonumber\\ &&+\int_0^x\int^x_\eta  d_{12}(\eta,\xi,t)
\d_tu_2(\eta,\om_1(\eta,\xi,t))\, d\xi d\eta.\label{eqv4}
\end{eqnarray}
The first summand meets the $C^1$-regularity. Let us show that this is the case  for the second summand.
On the account of the simple transformation
\begin{eqnarray*}
\d_\xi u_2(\eta,\om_1(\eta,\xi,t))=\d_2 u_2(\eta,\om_1(\eta,\xi,t))\d_\xi\om_1(\eta,\xi,t),
\end{eqnarray*}
where $\d_kg$ here and below  denotes the derivative of $g$ with respect to the $k$-th argument, we have
\begin{eqnarray}
\lefteqn{
\d_t u_2(\eta,\om_1(\eta,\xi,t))=\d_2 u_2(\eta,\om_1(\eta,\xi,t))\d_t\om_1(\eta,\xi,t)}\nonumber\\
&&=\frac{\d_t\om_1(\eta,\xi,t)}{\d_\xi\om_1(\eta,\xi,t)}\d_\xi u_2(\eta,\om_1(\eta,\xi,t)).\label{trans}
\end{eqnarray}
Here
\begin{eqnarray}
\label{dx}
\d_x\om_j(\xi,x,t) & = & \frac{(-1)^{j+1}}{a(x,t)} \exp \int_\xi^x(-1)^j
\left(\frac{\d_ta}{a^2}\right)(\eta,\om_j(\eta,x,t)) d \eta,
\\
\label{dt}
\d_t\om_j(\xi;x,t) & = & \exp 
\int_\xi^x (-1)^j\left(\frac{\d_ta}{a^2}\right)(\eta,\om_j(\eta,x,t)) d \eta.
\end{eqnarray}
Then in the notation 
$$
\tilde d_{12}(\eta,\xi,t)=d_{12}(\eta,\xi,t)\frac{\d_t\om_1(\eta,\xi,t)}{\d_\xi\om_1(\eta,\xi,t)}
$$
the second summand in \reff{trans} equals
\begin{eqnarray}
\lefteqn{
\int_0^x\int^x_\eta  \tilde d_{12}(\eta,\xi,t)
\d_\xi u_2(\eta,\om_1(\eta,\xi,t))\, d\xi d\eta}\nonumber\\
&&=\int_0^x \left[\tilde d_{12}(\eta,\xi,t)
 u_2(\eta,\om_1(\eta,\xi,t))\right]_{\xi=\eta}^x d\eta\nonumber\\ &&
-\int_0^x\int^x_\eta  \d_\xi\tilde d_{12}(\eta,\xi,t)
u_2(\eta,\om_1(\eta,\xi,t))\, d\xi d\eta.\label{eqv5}
\end{eqnarray}
We are prepared to conclude that the function $\d_t[S_2u_1](x,t)$ is continuously differentiable.
 Hence,  $[S_2u_1](x,t)$ has $C^2$-regularity in $t$. To prove that it has $C^2$-regularity also in $x$,
we follow a similar argument, but this time we differentiate \reff{fubini}
in $x$. 

The fact that $w$ given by \reff{w_sol} satisfies (\ref{1})--(\ref{3}) easily follows from 
(\ref{1s})--(\ref{3s}).

The same argument works also to show the equivalence of the problems  (\ref{1})--(\ref{3}) and 
(\ref{1s})--(\ref{3s}) in the sense of the weak solvability, the only difference being in 
applying the calculations performed by \reff{fubini}, \reff{eqv4}, \reff{trans}, and 
\reff{eqv5} with $u_2$ replaced by an arbitrary fixed sequence $u_2^l$
tending to $u_2$ in $C_T$ as $l\to\infty$. Passing to the limit  as $l\to\infty$ 
in thus obtained analog of \reff{eqv5} 
 finishes the proof.

\section{Fredholm alternative: proof of Theorem \ref{thm:Fredh}}\label{sec:Fredh}

On the account of Section \ref{sec:equival}, we are done if we prove 
the Fredholm alternative for (\ref{1s})--(\ref{3s}): First,
$\dim \KK_u <\infty$, where $\KK_u$ is the vector space of all  weak solutions to
\reff{1s}--\reff{3s} with $f=0$. Second,
the  space  of all  $f \in  C_T$ such that there exists a weak solution 
to~(\ref{1s})--(\ref{3s}) is a closed subspace of codimension $\dim \KK_u$ in $\CC$.
Third,  either  $\dim \KK_u>0$ or for any $f \in C_T$ 
there exists exactly one weak solution $u$
to~(\ref{1s})--(\ref{3s}).

To simplify further notation, in parallel with the notation $\om_j(\xi,x,t)$
we will use its shortened form $\om_j(\xi)$.
The system  \reff{rep1}--\reff{rep2} can be written as the operator equation
\beq\label{abstr}
u=Bu+Au+Du+Rf,
\ee
where the linear bounded operators $B,A,D: \CC\to \CC$ and  $R: C \to \CC$ are 
defined by
\begin{eqnarray}
[Bu](x,t)&=&\Bigl(c_1(0,x,t)u_2(0,\om_1(0)),c_2(1,x,t) u_1(1,\om_2(1))\Bigr)\nonumber\\
\label{Adef}
[Au](x,t)&=&\bigg(
2c_1(0,x,t)a(0,\om_1(0))r_0(\om_1(0))[Gu](\om_1(0)),\nonumber\\
&&-2c_2(1,x,t)a(1,\om_2(1))r_1(\om_2(1))[Fu](1,\om_2(1))
\bigg),\nonumber\\
\label{Ddef}
[Du](x,t)&=&\Bigg(
-\int_0^x d_1(\xi,x,t) \left(b_{12}u_2-[a_3Fu]\right)(\xi,\om_1(\xi))\,d\xi, \nonumber\\
&&
-\int_1^x d_2(\xi,x,t) \left(b_{21}u_1-[a_3Fu]\right)(\xi,\om_2(\xi))\,d\xi
\Bigg),\nonumber\\
\label{Gdef}
[Rf](x,t)&=&\left(\int_0^x d_1(\xi,x,t) f(\xi,\om_1(\xi))\,d\xi, 
\int_1^x d_2(\xi,x,t) f(\xi,\om_1(\xi))\,d\xi\right).\nonumber
\end{eqnarray}

We have to show that the operator $I-B-A-D$ is Fredholm of index zero from $\CC$ to~$\CC$.
First we prove the bijectivity of $I-B$:
\begin{lem}
\label{lem:iso}
If one of the conditions \reff{small} and \reff{small+} is fulfilled, then $I-B$ is bijective
from $\CC$ to $\CC$.
\end{lem}
\begin{proof}
Suppose \reff{small}. Let $g=(g_1,g_2) \in \CC$ be arbitrary given. 
We have $u=Bu+g$ or, the same,
\beq\label{iso0}
u_1(x,t)=c_1(0,x,t)u_2(0,\om_1(0))+g_1(x,t),\,\,\, u_2(x,t)=c_2(1,x,t) u_1(1,\om_2(1))+g_2(x,t)
\ee
 if and only if
\beq\label{iso1}
\begin{array}{ll}
u_1(x,t)=c_1(0,x,t)\left[c_2(1,0,\om_1(0)) u_1(1,\om_2(1,0,\om_1(0)))+g_2(0,\om_1(0))\right]+
g_1(x,t),\\
u_2(x,t)=c_2(1,x,t) u_1(1,\om_2(1))+g_2(x,t).
\end{array}
\ee
Observe that it suffices to show the existence of a unique continuous solution 
$t\in[0,T]\to u_1(1,t)\in\R$. Putting $x=1$ in  the first equation of \reff{iso1}, we get
\beq\label{iso2}
u_1(1,t)=c_1(0,1,t)c_2(1,0,\om_1(0,1,t)) u_1(1,\om_2(1,0,\om_1(0,1,t)))+\tilde g(1,t),
\ee 
where
$\tilde g(x,t)=c_1(0,x,t) g_2(0,\om_1(0))+
g_1(x,t)$. Putting then $t=\om_1\left(1,0,\om_2(0,1,\tau)\right)$, we come to another
writing of \reff{iso2}, namely
\beq\label{iso3}
\begin{array}{cc}
\left[c_1\left(0,1,\om_1\left(1,0,\om_2(0,1,\tau)\right)\right)
c_2(1,0,\om_2(0,1,\tau))\right]u_1(1,\tau)\\= 
\left(u_1-\tilde g\right)(1,\om_1(1,0,\om_2(0,1,\tau))).
\end{array}
\ee 
Here we used the identity $\om_2(1,0,\om_1(0,1,\om_1\left(1,0,\om_2(0,1,\tau)\right)))\equiv\tau$,
being true for all $\tau\in\R$.
Due to the Banach fixed point argument, Equations \reff{iso2} and
 \reff{iso3} are uniquely solvable under the contraction conditions, respectively,
$$
c_1(0,1,t)c_2(1,0,\om_1(0,1,t))<1 \mbox{ for all } t\in[0,T]
$$ 
and
$$
\left[c_1\left(0,1,\om_1\left(1,0,\om_2(0,1,\tau)\right)\right)
c_2(1,0,\om_2(0,1,\tau))\right]^{-1}<1 \mbox{ for all } \tau\in[0,T].
$$
Since the latter is equivalent to 
$\left[c_1(0,1,t)c_2(1,0,\om_1(0,1,t))\right]^{-1}<1$ for all $t\in[0,T]$, we immediately meet 
assumption \reff{small}. The proof under the assumption \reff{small} is thereby complete.

The proof  under the assumption  \reff{small+} follows along the same line as above, 
the only difference being in using instead of
\reff{iso1} another equivalent form of \reff{iso0}, namely
\beq
\begin{array}{ll}
u_1(x,t)=c_1(0,x,t)u_2(0,\om_1(0))+g_1(x,t),\nonumber\\
u_2(x,t)=c_2(1,x,t)\left[c_1(0,1,\om_2(1))u_2(0,\om_1(0,1,\om_2(1)))+g_1(1,\om_2(1))\right]+g_2(x,t),
\end{array}
\ee
and putting $x=0$ in the latter.
\end{proof}

Returning  to the operator $I-B-A-D$, we would like to emphasize that 
the operators $A$ and $D$ are  not compact from  $\CC$ to $\CC$, in general,
 because they are  partial integral operators
(other kinds of partial integral operators are investigated in \cite{Appell}).
By Lemma~\ref{lem:iso}, 
 the operator $I-B-A-D$ is Fredholm of index zero from $\CC$ to $\CC$
if and only if
$
I-(I-B)^{-1}(A+D) \mbox{ is Fredholm of index zero from } \CC \mbox{ to } \CC.
$
Then, on the account of Fredholmness criterion  \cite[Theorem XIII.5.2]{KA}, we are done if we prove 
the following statement:
\begin{lem}
\label{lem:comp}
The operator $\left[(I-B)^{-1}(A+D)\right]^2$ is compact from $\CC$ to $\CC$.
\end{lem}
\begin{proof}
Due to the boundedness of the operator $(I-B)^{-1}$,  it is sufficient to prove that
\beq
\label{comp} 
(A+D)(I-B)^{-1}(A+D) \mbox{ is compact from } \CC \mbox{ to } \CC.
\ee
 Since
$$
(A+D)(I-B)^{-1}(A+D)=(A+D)^2+(A+D)B(I-B)^{-1}(A+D),
$$  
the statement \reff{comp} will be proved if we show that
\beq
\label{Fr11}
(A+D)^2 \mbox{ and } (A+D)B  \mbox{ are compact from } \CC \mbox{ to } \CC.
\ee
By the Arzela-Ascoli theorem, $\CC^1$ 
is compactly embedded into $\CC$. Hence, 
for \reff{Fr11} it  suffices to show that
\beq
\label{Fr2}
(A+D)^2 \mbox{ and } (A+D)B  \mbox{ map  continuously } \CC \mbox{ into } \CC^1.
\ee
To reach \reff{Fr2}, we will prove the following smoothing property:
\beq
\label{Fr20}
A^2, D^2, AD, DA, AB, \mbox{ and } DB  \mbox{ map  continuously } \CC \mbox{ into } \CC^1.
\ee

Let us  start with the operator $A^2$. Using the definition \reff{Adef}, we are done if
we show that
\beq\label{Fr20A}
GA \mbox{ and } FA  \mbox{ map continuously } \CC \mbox{ into } \CC^1.
\ee
On the account of the  definition  \reff{G} of $G$ and the continuous embedding of $C_T^1(\R)$
into $\CC^1$, the operator $G$ and, hence,  the operator $GA$  maps continuously $\CC$
into  $\CC^1$. 
Moreover, by the definition  \reff{F} of $F$, to get \reff{Fr20A} for  $FA$
we only need to handle the operator
$$
[JAu](1,t)= \int_0^1\frac{[Au]_1(x,t)-[Au]_2(x,t)}{2a(x,t)}\,dx.
$$
Again, by the definition of $A$, we are left with the integral
$$
\int_0^1\frac{[Au]_2(x,t)}{2a(x,t)}\,dx
$$
or, even more, with the integral
\begin{eqnarray*}
\lefteqn{
\int_0^1\frac{c_2(1,x,t)a(1,\om_2(1))r_1(\om_2(1))}{2a(x,t)}[Ju](1,\om_2(1))\,dx}\\
&&=
\int_0^1\frac{c_2(1,x,t)a(1,\om_2(1))r_1(\om_2(1))}{2a(x,t)}\int_0^1
\left(\frac{u_1-u_2}{2a}\right)(\xi,\om_2(1))\, d\xi dx\\
&&=
\int_0^1\int_{\om_2(1,0,t)}^t
\frac{c_2(1,\tilde\om_2(t,1,\tau),t)a(1,\tau)r_1(\tau)\d_\tau\tilde\om_2(t,1,\tau)}
{2a(\tilde\om_2(t,1,\tau),t)}
\left(\frac{u_1-u_2}{2a}\right)(\xi,\tau)\, d\tau d\xi,
\end{eqnarray*}
where $\d_\tau\tilde\om_2(t,1,\tau)$ is given by \reff{tilde_dt}. The right-hand side of the latter 
equality has the desired smoothing property, what finishes the proof of \reff{Fr20A}.

Next we prove  \reff{Fr20} for $D^2$. 
Taking into account the density of  $\CC^1$ in $\CC$,
we are done if we show that there is a constant $C>0$ such that
\beq\label{D20}
\|\d_xD^2u\|_\infty+\|\d_tD^2u\|_\infty\le C\|u\|_\infty
\ee
for all $u\in\CC^1$.
Using the definitions of $D$ and $F$ and the smoothing property of $G$ mentioned above,
we only need to treat  integral operators of two types 
contributing into $D^2$. Thus, the integral operator of the first type 
\begin{eqnarray}
\lefteqn{
\int_0^xd_1(\xi,x,t)a_3(\xi,\om_1(\xi))\int_0^\xi\left(\frac{u_1-u_2}{2a}\right)(\eta,\om_1(\xi))
\,d\eta d\xi}\label{DF}\\
&&= -\int_{0}^x\int_{\om_1(\eta)}^t d_1(\tilde\om_1(\tau),x,t)
a_3(\tilde\om_1(\tau),\tau)a(\tilde\om_1(\tau),\tau)
\left(\frac{u_1-u_2}{2a}\right)(\eta,\tau)
\,d\tau d\eta, \nonumber
\end{eqnarray}
where $\tilde\om_1(\tau)=\tilde\om_1(\tau,x,t)$,
maps  continuously $\CC$  into $\CC^1$, what immediately entails the estimate of kind \reff{D20}
for it.
It remains to prove the upper bound $ C\|u\|_\infty$  
for the integral   of the type 
\beq\label{D21}
\begin{array}{ll}
\displaystyle\int_0^x\int_1^\xi d_{12}(\xi,\eta,x,t)u_1(\eta,\om_2(\eta;\xi,\om_1(\xi)))\, d \eta d \xi\\
\displaystyle=\int^0_x\int_0^\eta d_{12}(\xi,\eta,x,t)u_1(\eta,\om_2(\eta;\xi,\om_1(\xi)))\,d \xi d \eta \\
\displaystyle
+\int_x^1\int_0^x d_{12}(\xi,\eta,x,t)u_1(\eta,\om_2(\eta;\xi,\om_1(\xi)))\,d \xi d \eta
\end{array}
\ee
with
\begin{eqnarray*}
d_{12}(\xi,\eta,x,t)
=d_1(\xi,x,t)d_2(\eta,\xi,\om_1(\xi))b_{12}(\xi,\om_1(\xi))
b_{21}(\eta,\om_2(\eta;\xi,\om_1(\xi))).
\label{d12}
\end{eqnarray*}
Note that
\begin{eqnarray*}
&\displaystyle
(\d_t-a(x,t)\d_x)\int_0^x\int_1^\xi d_{12}(\xi,\eta,x,t)u_1(\eta,\om_2(\eta;\xi,\om_1(\xi)))\, d \eta d \xi &\\
&\displaystyle
=-a(x,t)\int_1^xd_{12}(x,\eta,x,t)u_1(\eta,\om_2(\eta))\, d \eta,&
\end{eqnarray*}
where the derivatives are considered in a distributional sense.
Hence, to  derive \reff{D20} with $[D^2u](x,t)$ replaced by \reff{D21}, it is sufficient to prove the estimate
$\|\d_tD^2u\|_\infty\le C\|u\|_\infty$ satisfying uniformly in  $u\in\CC^1$. Thus,
we differentiate \reff{D21}
with respect to $t$  (without loss of generality 
we illustrate our argument only on the first summand in the right-hand side of \reff{D21}) and get
\begin{eqnarray}
\lefteqn{
\int_0^x\int_0^\eta \d_t d_{12}(\xi,\eta,x,t)u_1(\eta,\om_2(\eta;\xi,\om_1(\xi)))\, d\xi d\eta}
\nonumber\\
&&+\int_0^x\int_0^\eta  d_{12}(\xi,\eta,x,t)\d_t u_1(\eta,\om_2(\eta;\xi,\om_1(\xi)))\, d\xi d\eta.\label{D22}
\end{eqnarray}
The first summand obviously fits the desired estimate. To handle the second one,
we compute
\begin{eqnarray*}
\lefteqn{
\d_\xi u_1(\eta,\om_2(\eta;\xi,\om_1(\xi)))}\\ &&= \d_2 u_1(\eta,\om_2(\eta;\xi,\om_1(\xi)))
\left[\d_2\om_2(\eta;\xi,\om_1(\xi))+\d_3\om_2(\eta;\xi,\om_1(\xi))\d_\xi\om_1(\xi) \right].
\end{eqnarray*}
Hence, applying \reff{char},  \reff{dx}, and \reff{dt} gives
\beq
\begin{array}{cc}
\d_t u_1(\eta,\om_2(\eta;\xi,\om_1(\xi)))
=\displaystyle\d_2 u_1(\eta,\om_2(\eta;\xi,\om_1(\xi)))\d_3\om_2(\eta;\xi,\om_1(\xi))\d_t\om_1(\xi)
\nonumber\\ [2mm]
=\displaystyle
\frac{\d_\xi u_1(\eta,\om_2(\eta;\xi,\om_1(\xi)))\d_3\om_2(\eta;\xi,\om_1(\xi))\d_t\om_1(\xi)}
{\d_2\om_2(\eta;\xi,\om_1(\xi))+\d_3\om_2(\eta;\xi,\om_1(\xi))\d_\xi\om_1(\xi)}
\nonumber\\ [3mm]\displaystyle
=-\frac{1}{2}a(\xi,\om_1(\xi))\d_t\om_1(\xi)
\d_\xi u_1(\eta,\om_2(\eta;\xi,\om_1(\xi))).\nonumber
\end{array}
\ee
Then, using the notation
$$
\tilde d_{12}(\xi,\eta,x,t)=-\frac{1}{2}d_{12}(\xi,\eta,x,t)a(\xi,\om_1(\xi))\d_t\om_1(\xi),
$$
the second summand in \reff{D22} equals
\begin{eqnarray*}
\lefteqn{
\int_0^x\int_0^\eta  \tilde d_{12}(\xi,\eta,x,t)\d_\xi u_1(\eta,\om_2(\eta;\xi,\om_1(\xi)))\, 
d \xi d \eta}\\
&&=\int_0^x\left[\tilde d_{12}(\xi,\eta,x,t) u_1(\eta,\om_2(\eta;\xi,\om_1(\xi)))
\right]_{\xi=0}^{\xi=\eta}\,  d \eta\\
&&-\int_0^x\int_0^\eta \d_\xi \tilde d_{12}(\xi,\eta,x,t) u_1(\eta,\om_2(\eta;\xi,\om_1(\xi)))\, 
d \xi d \eta,
\end{eqnarray*}
what immediately entails the desired estimate. We therefore finished with the
estimate \reff{D20}.

Further we prove \reff{Fr20} for the operator $AD$. As above, due to the definition of $A$,
we are reduced to give the proof  for the operator $FD$ only.
On the account of the definition of $F$, the latter will be proved once we handle the operator
$JD$.
Thus,
\begin{eqnarray*}
[JDu](x,t)=&-& \int_0^x\frac{1}{2a(\xi,t)}\int_0^\xi d_1(\eta,\xi,t)
\Bigl(b_{12}u_2-[a_3Fu]\Bigr)(\eta,\om_1(\eta,\xi,t))\,d\eta d\xi \\
&+& \int_0^x\frac{1}{2a(\xi,t)}\int_1^\xi d_2(\eta,\xi,t)
\Bigl(b_{21}u_1-[a_3Fu]\Bigr))(\eta,\om_2(\eta,\xi,t))\,d\eta d\xi.
\end{eqnarray*}
After  changing the order of integration and making 
a simple change of variables the first summand in the right-hand side
(and similarly for the second summand)  can be written in the form 
\begin{eqnarray*}
- \int_0^x\int_t^{\om_1(\eta)}\frac{1}{2a(\tilde\om_1(t,\eta,\tau),t)} 
d_1(\eta,\tilde\om_1(t,\eta,\tau),t)
\Bigl(b_{12}u_2-[a_3Fu]\Bigr)(\eta,\tau)\,d\tau d\eta
\end{eqnarray*}
allowing to state the desired smoothing property.

On the next step we treat the operator $DA$. For instance, for  $[DAu]_1$
 (and similarly for $[DAu]_2$) we have
\begin{eqnarray*}
\lefteqn{
[DAu]_1(x,t)=
 \int_0^x d_1(\xi,x,t)
\Bigl(2b_{12}(\xi,\om_1(\xi))c_2(1,\xi,\om_1(\xi))a(1,\om_2(1,\xi,\om_1(\xi)))}\\
&&\times
r_1(\om_2(1,\xi,\om_1(\xi)))[Fu](1,\om_2(1,\xi,\om_1(\xi)))
-[a_3FAu]\Bigr)(\xi,\om_1(\xi))\,d\xi.
\end{eqnarray*}
Again, by the definition of $F$, we are done if we prove the smoothing property \reff{D20}
for the latter expression but with $J$ in place of $F$. Here
one can apply the same argument as in \reff{DF} (changing the order of integration and using the
 changing of variables $\tau=\om_1(\xi)$).

Turning back to \reff{Fr20}, we further proceed with the operator $AB$. By the definition
of  $A, B$, and  $F$, 
it suffices to show that the operator $JB$ maps  continuously $\CC$  into $C_T^1$. 
Indeed,
\beq
\begin{array}{cc}
\displaystyle[JBu](x,t)=
 \int_0^x 
\frac{c_1(0,\xi,t)u_2(0,\om_1(0,\xi,t))-c_2(1,\xi,t)u_1(1,\om_2(1,\xi,t))}{2a(\xi,t)}\,d\xi
\nonumber\\\displaystyle
= \int_t^{\om_1(0)}
\frac{c_1(0,\tilde\om_1(t,0,\tau),t)u_2(0,\tau)}
{2a(\tilde\om_1(t,0,\tau),t)}\d_\tau\tilde\om_1(t,0,\tau)\,d\tau\nonumber\\
\displaystyle
- \int_{\om_2(1,0,t)}^{\om_2(1)}
\frac{c_2(1,\tilde\om_2(t,1,\tau),t)u_1(1,\tau)}
{2a(\tilde\om_2(t,1,\tau),t)}\d_\tau\tilde\om_2(t,1,\tau)\,d\tau.\nonumber
\end{array}
\ee
The desired property  for $AB$ now easily follows from the smoothness assumptions 
\reff{smooth} and \reff{tilde_dx}.  

Finally, we prove \reff{Fr20} for  the operator $DB$. 
Denote by $\tilde x(\tau,x,t)$ the value of $\xi$ 
where the characteristics $\om_2(\xi,1,\tau)$ and $\om_1(\xi,x,t)$ intersect, namely
$$
\om_2(\tilde x(\tau,x,t),1,\tau)=\om_1(\tilde x(\tau,x,t),x,t).
$$
It follows from \reff{smooth}  that the function $\tilde x(\tau,x,t)$ is continuously differentiable
in its arguments.
Furthermore,
\begin{eqnarray}
\d_\tau\tilde x(\tau,x,t)&=&\frac{\d_3\om_2(\tilde x(\tau,x,t),1,\tau)}
{\d_1\om_1(\tilde x(\tau,x,t),x,t)-\d_1\om_2(\tilde x(\tau,x,t),1,\tau)} 
\label{x_tilde_tau}\\
&
=&-\frac{a\left(\tilde x(\tau,x,t),\om_1(\tilde x(\tau,x,t))\right)}{2}
\exp \int_ {\tilde x(\tau,x,t)}^1 
\left(\frac{\d_ta}{a^2}\right)(\eta,\om_2(\eta;1,\tau)) 
\,d \eta.\nonumber
\end{eqnarray}
Here we used \reff{char} and \reff{dx}. Similarly,
\begin{eqnarray}
\d_x\tilde x(\tau,x,t)=
\frac{a\left(\tilde x(\tau,x,t),\om_1(\tilde x(\tau,x,t))\right)}{2a(x,t)}
\exp \int_x^{\tilde x(\tau,x,t)} 
\left(\frac{\d_ta}{a^2}\right)(\eta,\om_1(\eta)) 
\,d \eta,\label{x_tilde_x}\\
\d_t\tilde x(\tau,x,t)=
\frac{a\left(\tilde x(\tau,x,t),\om_1(\tilde x(\tau,x,t))\right)}{2}
\exp \int_x^{\tilde x(\tau,x,t)} 
\left(\frac{\d_ta}{a^2}\right)(\eta,\om_1(\eta))
\,d \eta.\label{x_tilde_t}
\end{eqnarray}

By the definitions of $D$ and $B$ as well as \reff{DF}, to handle  $DB$,
it remains to treat the integrals of the type
\begin{eqnarray*}
\lefteqn{
\int_0^x d_1(\xi,x,t) b_{12}(\xi,\om_1(\xi))c_2(1,\xi,\om_1(\xi))
u_1(1,\om_2(1,\xi,\om_1(\xi)))\,d\xi
}\\
&&
=\int_{\om_2(1,0,\om_1(0))}^{\om_2(1)} d_1(\tilde x(\tau,x,t),x,t) 
b_{12}\left(\tilde x(\tau,x,t),\om_1(\tilde x(\tau,x,t))\right)\\
&&\times
c_2\left(1,\tilde x(\tau,x,t),\om_1(\tilde x(\tau,x,t))\right)
u_1(1,\tau)
\d_\tau\tilde x(\tau,x,t)
\,d\tau.
\end{eqnarray*}
On the account of \reff{x_tilde_tau}, \reff{x_tilde_x}, and \reff{x_tilde_t},
we immediately come to the desired conclusion, what completes the proof
of the lemma.
\end{proof}

\section{Higher regularity of solutions: proof of Theorem \ref{thm:reg}}

Here we address the issue  of a higher  regularity of  weak solutions in the case 
of a higher regularity of the coefficients 
in \reff{1} and \reff{3} and an additional number of non-resonance conditions. 
We therefore let \reff{a}, \reff{ar}, \reff{k-smooth}, and one of the 
conditions \reff{small1}, \reff{small11}, \reff{small111}, and \reff{small1111} to be fulfilled.

First note that the statement $(ii)$ of Theorem \ref{thm:reg} is a straightforward 
consequence of the statement $(i)$, since in the case of stationary $a$ 
we have $c_j^l\equiv c_j$ for all $l\ge 1$ and, hence, 
 condition \reff{small} implies either \reff{small1} or  \reff{small11} 
for any positive integer $k$.

On the account of the equivalence of the problems  \reff{1}--\reff{3} and  \reff{1s}--\reff{3s}
stated in Section \ref{sec:equival}, we are reduced 
 to prove that any weak solution $u$ to  \reff{1s}--\reff{3s}
reaches $C^{k}$-regularity. To this end, we introduce a couple of Banach spaces: 
Given a positive integer $i$, set
$$
\tilde C_T^i=\{u\in C_T\, :\,\d_t^iu\in C_T\} \quad\mbox{ and }\quad \tilde\CC^i=\tilde C_T^i\times\tilde  C_T^i.
$$

Let $u\in\CC$ be an arbitrary fixed weak solution to  \reff{1s}--\reff{3s}. The proof 
is by induction on the order of regularity
(the order of continuous differentiability) of the solutions. 

{\bf Base case:} {\it $u\in\CC^1$.} 
First show that the generalized directional derivatives $(\d_t-a\d_x)u_1$ and $(\d_t+a\d_x)u_2$, where 
 $\d_x$ and $\d_t$ denote the generalized derivatives,  is a continuous function; this reduces our task to proving
 that  $u\in\tilde\CC^1$. Take an arbitrary sequence $u^l\in\CC^{1}$ approaching
$u$ in $\CC$ and an arbitrary smooth function 
$\vphi: (0,1)\times(0,T)\to\R$ 
with compact support. Then
\begin{eqnarray*}
\lefteqn{
\langle (\d_t-a\d_x)u_1,\vphi\rangle = \langle u_1,-\d_t\vphi+\d_x(a\vphi)\rangle
= \lim_{l\to\infty}\left\langle u_1^l,-\d_t\vphi+\d_x(a\vphi)\right\rangle} \\ &&=
\lim_{l\to\infty}\Biggl\langle 
c_1(0,x,t)\Big[u_2^l(0,\om_1(0))
+2a(0,\om_1(0))r_0(\om_1(0))[Gu^l](\om_1(0))\Big]
\\
&&-\int_0^x d_1(\xi,x,t) b_{12}(\xi,\om_1(\xi))u_2^l(\xi,\om_1(\xi))d\xi\\ 
&&+\int_0^x d_1(\xi,x,t)\Bigl(f(\xi,\om_1(\xi))-[a_3Fu^l](\xi,\om_1(\xi))\Bigr)\,d\xi,
-\d_t\vphi+\d_x(a\vphi)\Biggr\rangle \\ &&=
\lim_{l\to\infty}\left\langle -b_{11}(x,t)u_1^l- b_{12}(x,t)u_2^l + f(x,t)-[a_3Fu^l](x,t),\vphi\right\rangle 
\\ &&
=\left\langle -b_{11}(x,t)u_1- b_{12}(x,t)u_2 + f(x,t)-[a_3Fu](x,t),\vphi\right\rangle
\end{eqnarray*}
as desired.
Here we used the formula
$$
(\d_t+(-1)^ja(x,t)\d_x)\psi(\om_j(\xi,x,t))=0
$$
being true for all $j=1,2$, $\xi,x\in[0,1]$, $t\in\R$, and any $\psi\in C^1(\R)$.
 Similarly we compute the generalized directional derivative 
$(\d_t+a\d_x)u_2$. 

Therefore, the beginning step of the induction will follow from the fact 
that $u\in\tilde\CC^1$. To prove the latter,
we substitute \reff{abstr}  into the second and the third summands of
\reff{abstr} and get

\beq\label{final}
u=Bu+\left(A^2+AB+AD+DB+DA+D^2\right)u+(I+A+D)Rf.
\ee
On the account of the smoothing property \reff{Fr20} of the operators
$A^2$, $AB$, $AD$ $DB$, $DA$, and $D^2$, we are done if we 
show the bijectivity of $I-B\in\LL\left(\CC\right)$ from $\tilde\CC^1$ to $\tilde\CC^1$. 
In other words, we have to show that the system \reff{iso1}
is uniquely solvable in $\tilde\CC^1$ for any $(g_1,g_2)\in\tilde\CC^1$.
Following the argument as in the proof of Lemma~\ref{lem:iso},
the latter is true iff
\beq\label{contr1}
I-B^\prime \mbox{ is bijective from }   C_T^1(\R) \mbox{ to }   C_T^1(\R),
\ee
where the operator $B^\prime\in\LL\left(C_T(\R)\right)$ is given by
\beq
\left[B^\prime v\right](t)=c_1(0,1,t)c_2(1,0,\om_1(0,1,t)) v(\om_2(1,0,\om_1(0,1,t))).\nonumber
\ee

Now we aim to show \reff{contr1} whenever one of  the 
 conditions \reff{small1}, \reff{small11}, \reff{small111}, and \reff{small1111} with $k=1$ is satisfied.
We will prove the desired statement under the condition \reff{small1}.
A similar argument combined with the one used in the proof of Lemma~\ref{lem:iso}
works in the case of  \reff{small11}, \reff{small111}, or \reff{small1111} as well.

Thus, following the idea of  \cite{RauchReed81} (see also \cite{KR3})
used to establish  the solution regularity  for
first-order hyperbolic PDEs, given  $\be>0$, we  norm the space  $C_T^1(\R)$ with
\beq\label{beta}
\|v\|_{C_T^1(\R)}=\|v\|_\infty+\be\|\d_tv\|_\infty.
\ee
Note that  $C_T^1(\R)$ endowed with \reff{beta} is a Banach space.
Given $l=0,1,2,\dots$, set 
\begin{eqnarray}
q_l&=&\max_{x,y\in[0,1]}\max_{t\in\R}\left|c_1^l(0,1,t)c_2^l(1,0,\om_1(0,1,t))\right|,\nonumber\\
q_l^\prime&=&\max_{x,y\in[0,1]}\max_{t\in\R}\left|\frac{d}{dt}\left[c_1^l(0,1,t)c_2^l(1,0,\om_1(0,1,t))\right]\right|.
\label{ql_prime}\end{eqnarray}
We are reduced to prove that  there exist  constants $\be<1$ and $\ga<1$ such that
$$
\|B^\prime v\|_\infty+\be\left\|\frac{d}{dt}B^\prime v\right\|_\infty 
\le \ga\left(\|v\|_\infty+\be\|v^\prime\|_\infty\right) 
\mbox{ for all } v \in  C_T^1(\R).
$$
Taking into account that $\|B^\prime\|_{\LL(C_T(\R))}\le q_0<1$ (by assumption \reff{small}), 
the latter estimate will be proved if we show that
\beq
\label{cungl0}
\left\|\frac{d}{dt}B^\prime v\right\|_\infty \le \frac{\ga-q_0}{\be}\|v\|_\infty
+\ga\|v^\prime\|_\infty\quad \mbox{ for all } v\in  C_T^1(\R).
\ee
Since 
\begin{eqnarray*}
\label{C_t}
\lefteqn{
\frac{d}{dt}[B^\prime v](t)=\frac{d}{dt}\left[c_1(0,1,t)c_2(1,0,\om_1(0,1,t))\right] v(\om_2(1,0,\om_1(0,1,t)))}\\
&&+c_1(0,1,t)c_2(1,0,\om_1(0,1,t))\d_2\om_2(1,0,\om_1(0,1,t))\d_t\om_1(0,1,t)v^\prime(\om_2(1,0,\om_1(0,1,t))),
\end{eqnarray*}
we get
\beq
\label{510}
\left\|\frac{d}{dt}B^\prime v\right\|_\infty \le q_0^\prime\|v\|_\infty+q_1\|v^\prime\|_\infty.
\ee
By the assumption \reff{small1} we have $q_0<1$ and  $q_1<1$. Fix $\ga$ such that
$
\max\{q_0,q_1\} <\ga<1.
$
Then choose $\be$ so small that 
$$
q_0^\prime\le \frac{\ga-q_0}{\be}. 
$$
Then \reff{510} implies \reff{cungl0} as desired. 
The proof of the base case of the induction is therewith complete.
Notice that  the function $u$ now satisfies the system 
\reff{1s} pointwise. 

{\bf Induction assumption:}
{\it $u\in \CC^i$ for some $1\le i\le k-1$.}

{\bf Induction step:}
{\it $u\in \CC^{i+1}$.} By the induction assumption, the function $u$ satisfies the following 
system pointwise:
\beq\label{1s_i}
\begin{array}{lr}
\displaystyle\left(\partial_t  - a\partial_x\right)\d_t^{i-1}u_1 = 
- \left(b_{11} + (i-1)\frac{\d_ta}{a}\right)\d_t^{i-1}u_1 - b_{12}\d_t^{i-1}u_2 \\[3mm]
\displaystyle\qquad\quad + 
f_{1,i-1}\left(x,t,u,\d_tu,\dots,\d_t^{i-2}u\right)-\left[P_{i-1}u\right](x,t)
-\left[a_3J\d_t^{i-1}u\right](x,t), 
 \\[5mm]
\displaystyle
\left(\partial_t  + a\partial_x\right)\d_t^{i-1}u_2 =
- b_{21}\d_t^{i-1}u_1 - \left(b_{22} - (i-1)\frac{\d_ta}{a}\right)\d_t^{i-1}u_2  \\[3mm]
\displaystyle\qquad\quad + 
f_{2,i-1}\left(x,t,u,\d_tu,\dots,\d_t^{i-2}u\right)-\left[P_{i-1}u\right](x,t)
-\left[a_3J\d_t^{i-1}u\right](x,t)
\end{array}
\ee
with certain continuously differentiable functions $f_{1,i-1}$ and $f_{2,i-1}$ such that
$f_{1,0}=f_{2,0}\equiv 0$ and with the operators $P_i\in\LL(\CC^i)$ defined by
\beq
\begin{array}{cc}
\displaystyle [P_iu](x,t) = \d_t^ia_3\left[Iu\right](t)+\frac{1}{2}\d_t^{i-1}
\left[a_3(u_1(0,t)+u_2(0,t))\right]\nonumber\\\displaystyle+
\d_t^{i-1}\int_0^x\d_t\left(\frac{a_3(x,t)}{a(\xi,t)}\right)
\frac{u_1(\xi,t)-u_2(\xi,t)}{2}\,d\xi\nonumber
\end{array}
\ee
such that $[P_0u](x,t)\equiv 0$.
Using  \reff{k-smooth} and the induction assumption, we see that
the right-hand side of \reff{1s_i} is continuously differentiable in $t$.
Hence, the left-hand side is continuously differentiable in $t$ as well.
Note that the latter does not imply the existence of the pointwise derivatives
$\d_t^{i+1}u_j$ and $\d_x\d_t^{i}u_j$ for $j=1,2$, but only the distributional ones.
Set $v=\d_t^{i}u$. Then the continuous function $v$ satisfies the system
\beq\label{1s_i+1}
\begin{array}{lr}
\displaystyle\left(\partial_t  - a\partial_x\right)v_1 = 
- \left(b_{11} + i\frac{\d_ta}{a}\right)v_1 - b_{12}v_2 \\[3mm]
\displaystyle\qquad\quad + 
f_{1i}\left(x,t,u,\d_tu,\dots,\d_t^{i-1}u\right)-\left[P_iu\right](x,t)-
\left[a_3Jv\right](x,t), 
 \\[5mm]
\displaystyle
\left(\partial_t  + a\partial_x\right)v_2 =
- b_{21}v_1 - \left(b_{22} - i\frac{\d_ta}{a}\right)v_2  \\[3mm]
\displaystyle\qquad\quad + 
f_{2i}\left(x,t,u,\d_tu,\dots,\d_t^{i-1}u\right)-\left[P_iu\right](x,t)
-\left[a_3Jv\right](x,t)
\end{array}
\ee
in a distributional sense and the conditions 
\beq\label{2s_i+1}
v_j(x,t) = v_j(x,t+T),\quad j=1,2,
\ee
and
\beq\label{3s_i+1}
\begin{array}{llr}
\displaystyle
v_1(0,t) =\displaystyle v_2(0,t)+2\d_t^i\left(a(0,t)r_0(t)\right)[Gu](t)+
\d_t^{i-1}\left[a(0,t)r_0(t)(u_1(0,t)+u_2(0,t))\right],\\
\displaystyle
v_2(1,t) =\displaystyle v_1(1,t)-
 2\d_t^i(a(1,t)r_1(t))\left[Gu\right](t)-
\d_t^{i-1}\left[a(1,t)r_1(t)(u_1(0,t)+u_2(0,t))\right]\\ 
\displaystyle-\d_t^{i-1}\int_0^x\d_t\left(\frac{a(1,t)r_1(t)}{a(\xi,t)}\right)
\left(u_1(\xi,t)-u_2(\xi,t)\right)\,d\xi
-2a(1,t)r_1(t)\left[Jv\right](1,t)
\end{array}
\ee
pointwise. We rewrite the system \reff{1s_i+1}--\reff{3s_i+1} in the following
form:
\beq\label{1s_i+1_short}
\begin{array}{lr}
\displaystyle\left(\partial_t  - a\partial_x\right)v_1 = 
- \left(b_{11} + i\frac{\d_ta}{a}\right)v_1 - b_{12}v_2 
- \left[a_3Jv\right](x,t) + \left[Q_iu\right]_1(x,t), 
 \\[5mm]
\displaystyle
\left(\partial_t  + a\partial_x\right)v_2 =
- b_{21}v_1 - \left(b_{22} - i\frac{\d_ta}{a}\right)v_2  
-\left[a_3Jv\right](x,t) + \left[Q_iu\right]_2(x,t),
\end{array}
\ee
\beq\label{2s_i+1_short}
v_j(x,t) = v_j(x,t+T),\quad  j=1,2,
\ee
\beq\label{3s_i+1_short}
\begin{array}{rcl}
\displaystyle
v_1(0,t) &=&\displaystyle v_2(0,t)+\left[S_{i}u\right]_1(t),\\
\displaystyle
v_2(1,t) &=&\displaystyle v_1(1,t)+\left[S_{i}u\right]_2(t)
-2a(1,t)r_1(t)\left[Jv\right](1,t),
\end{array}
\ee
where the operators $Q_i,S_{i}\in\LL(\CC^i)$ are defined 
by 
\begin{eqnarray}
 \left[Q_iu\right](x,t)&=&\Bigl(f_{1i}\left(x,t,u,\d_tu,\dots,\d_t^{i-1}u\right)-\left[P_iu\right](x,t),
\nonumber\\ &&
f_{2i}\left(x,t,u,\d_tu,\dots,\d_t^{i-1}u\right)-\left[P_iu\right](x,t)\Bigr),\label{NQ}\\
\left[S_{i}u\right](t)&=&\Biggl(2\d_t^i\left(a(0,t)r_0(t)\right)[Gu](t)+
\d_t^{i-1}\left(a(0,t)r_0(t)(u_1(0,t)+u_2(0,t))\right),\nonumber\\
&&-
 2\d_t^i(a(1,t)r_1(t))\left[Gu\right](t)-
\d_t^{i-1}\left[a(1,t)r_1(t)(u_1(0,t)+u_2(0,t))\right]\nonumber\\
&&-\d_t^{i-1}\int_0^x\d_t\left(\frac{a(1,t)r_1(t)}{a(\xi,t)}\right)
\left(u_1(\xi,t)-u_2(\xi,t)\right)\,d\xi\Biggr).\label{NS}
\end{eqnarray}
It follows, in particular, that
\beq\label{*}
\begin{array}{ll}
 \mbox{ Given } u\in\CC^i,  
 \mbox{ the functions } \left[Q_iu\right]_1(x,t),  \left[Q_iu\right]_2(x,t),\\
\left[S_{i}u\right]_1(t),  \mbox{ and }
\left[S_{i}u\right]_2(t)  \mbox{ are continuously differentiable.}
\end{array}
\ee
We intend to show that the variational problem \reff{1s_i+1_short}--\reff{3s_i+1_short}
is equivalent to the following integral system:
\begin{eqnarray}
\label{rep1_i}
v_1(x,t)&=& c_1^i(0,x,t)\Big[v_2(0,\om_1(0))
+[S_{i}u]_1(\om_1(0))\Big]
\nonumber\\
&&-\int_0^x d_1^i(\xi,x,t) \Bigl(b_{12}v_2+[a_3Jv]-[Q_iu]_1\Bigr)
(\xi,\om_1(\xi))\,d\xi,
\end{eqnarray}
\begin{eqnarray}
\label{rep2_i}
v_2(x,t)&=&c_2^i(1,x,t)\Bigl[
 v_1(1,\om_2(1))
+[S_{i}u]_2(\om_2(1))
-2r_1(\om_2(1))\left[aJv\right](1,\om_2(1))\Bigr]
\nonumber\\
&&-\int_1^x d_2^i(\xi,x,t) \Bigl(b_{21}v_1-[a_3Jv]+[Q_iu]_2\Bigr)
(\xi,\om_2(\xi))\,d\xi,
\end{eqnarray}
where
\begin{eqnarray}
d_j^i(\xi,x,t)=\frac{(-1)^jc_j^i(\xi,x,t)}{a(\xi,\om_j(\xi,x,t))}.\nonumber
\end{eqnarray}

In other words, any function $u\in\CC^i$ satisfies \reff{1s_i+1_short}--\reff{3s_i+1_short}
in a distributional sense if and only if $u$ satisfies \reff{rep1_i}--\reff{rep2_i}
pointwise. 

To show the sufficiency, take an arbitrary sequence $u^l\in\CC^{i+1}$ approaching
$u$ in $\CC^i$ and write $v^l=\d_t^iu^l$. Then, taking into account \reff{*}, for any smooth function 
$\vphi: (0,1)\times(0,T)\to\R$ 
with compact support we have
\begin{eqnarray*}
\lefteqn{
\langle (\d_t-a\d_x)v_1,\vphi\rangle = \langle v_1,-\d_t\vphi+\d_x(a\vphi)\rangle
= \lim_{l\to\infty}\left\langle v_1^l,-\d_t\vphi+\d_x(a\vphi)\right\rangle} \\ &&=
\lim_{l\to\infty}\Biggl\langle 
c_1^i(0,x,t)\Big[v_2^l(0,\om_1(0))
+[S_{i}u]_1(\om_1(0))\Big]
\\
&&-\int_0^x d_1^i(\xi,x,t) b_{12}(\xi,\om_1(\xi))v_2^l(\xi,\om_1(\xi))d\xi\\ 
&&-\int_0^x d_1^i(\xi,x,t)\Bigl([a_3Jv^l]-[Q_iu]_1\Bigr)(\xi,\om_1(\xi))\,d\xi,
-\d_t\vphi+\d_x(a\vphi)\Biggr\rangle \\ &&=
\lim_{l\to\infty}\left\langle - \left(b_{11} + i\frac{\d_ta}{a}\right)v_1^l - b_{12}v_2^l
- \left[a_3Jv^l\right](x,t) + \left[Q_iu\right]_1(x,t),\vphi\right\rangle \\ &&=
\left\langle - \left(b_{11} + i\frac{\d_ta}{a}\right)v_1 - b_{12}v_2
- \left[a_3Jv\right](x,t) + \left[Q_iu\right]_1(x,t),\vphi\right\rangle.
\end{eqnarray*}
 Similarly we compute the generalized directional derivative 
$(\d_t+a(x,t)\d_x)u_2$. The sufficiency is thereby proved.

To show the necessity, assume that $u\in\CC^i$ satisfies 
\reff{1s_i+1_short}--\reff{3s_i+1_short} in a distributional sense.
Without destroying the equalities in $\D^\prime$, we  rewrite the system
 \reff{1s_i+1_short} in the form 
\beq
\begin{array}{lr}
\displaystyle\left(\partial_t  - a\partial_x\right) \left(c_1^i(1,x,t)v_1\right) = 
 c_1^i(1,x,t)\Bigl(- b_{12}v_2 
- \left[a_3Jv\right](x,t) + \left[Q_iu\right]_1(x,t)\Bigr), \nonumber
 \\[5mm]
\displaystyle
\left(\partial_t  + a\partial_x\right)\left(c_2^i(0,x,t)v_2\right) =
c_2^i(0,x,t)\Bigl(- b_{21}v_1  - \left[a_3Jv\right](x,t) + \left[Q_iu\right]_2(x,t)\Bigr).\nonumber
\end{array}
\ee
To prove that $v$ satisfies \reff{rep1_i}--\reff{rep2_i}
pointwise, we use  the constancy theorem of distribution theory claiming
that any distribution on an open set with zero generalized derivatives
is a constant on any connected component of the set. Hence, the
sums
\begin{eqnarray*}
v_1(x,t)+\int_0^x d_1^i(\xi,x,t)\Bigl(b_{12}+[a_3Jv]-[Q_iu]_1\Bigr)(\xi,\om_1(\xi))\,d\xi
\end{eqnarray*}
and
\begin{eqnarray*}
v_2(x,t)+\int_1^x d_2^i(\xi,x,t)\Bigl(b_{21}+[a_3Jv]-[Q_iu]_2\Bigr)(\xi,\om_2(\xi))\,d\xi
\end{eqnarray*}
are constants along the characteristics $\om_1(\xi,x,t)$ and $\om_2(\xi,x,t)$,
respectively. Since they are  continuous functions and the traces 
$v_1(0,t)$ and $v_2(1,t)$ are given by \reff{3s_i+1_short}, it follows that
 $v$ satisfies the system \reff{rep1_i}--\reff{rep2_i} as desired.

We are therefore reduced to prove that the function $v$ satisfying the system
\reff{rep1_i}--\reff{rep2_i} is continuously differentiable. To this end, 
for $i\ge 1$ we introduce
linear bounded operators $B_i,A_i,D_i,R_i: \CC \to \CC$ 
by
\begin{eqnarray}
&&[B_iv](x,t)=\Bigl(c_1^i(0,x,t)v_2(0,\om_1(0)),c_2^i(1,x,t) v_1(1,\om_2(1))\Bigr)\nonumber\\
\label{Aidef}
&&[A_iv](x,t)=\bigg(0,-2c_2^i(1,x,t)r_1(\om_2(1))[aJv](1,\om_2(1))\bigg),\\
&&[D_iu](x,t)=\Bigg(
-\int_0^x d_1^i(\xi,x,t) \left(b_{12}v_2-[a_3Jv]\right)(\xi,\om_1(\xi))\,d\xi, \nonumber\\
&&\qquad\qquad-\int_1^x d_2^i(\xi,x,t) \left(b_{21}v_1-[a_3Jv]\right)(\xi,\om_2(\xi))\,d\xi
\Bigg),\nonumber\\
\label{Gidef}
&&[R_iu](x,t)= \Bigg(c_1^i(0,x,t)[S_{i}u]_1(\om_1(0))
+\int_0^x d_1^i(\xi,x,t) \left[Q_iu\right]_1
(\xi,\om_1(\xi))\,d\xi, \nonumber\\
&&\qquad\qquad c_2^i(1,x,t)[S_{i}u]_2(\om_2(1))
-\int_1^x d_2^i(\xi,x,t)\left[Q_iu\right]_2
(\xi,\om_2(\xi))\,d\xi\Bigg).
\end{eqnarray}
and rewrite \reff{rep1_i}--\reff{rep2_i} in the operator form
\beq\label{abstr_i}
v=B_iv+A_iv+D_iv+R_iu.
\ee
Similarly to the base case of the induction, we will use the following equation for $v$
(the analog of \reff{final})
\beq
v=B_iv+\left(A_i^2+A_iB_i+A_iD_i+D_iB_i+D_iA_i+D_i^2\right)v+(I+A_i+D_i)R_iu,\nonumber
\ee
resulting from \reff{abstr_i}, and prove that it is uniquely solvable in $\tilde\CC^1$.
To this end, it is sufficient to show that
\beq\label{isom}
I-B_i \mbox{  is bijective from } \tilde\CC^1 \mbox{  to } \tilde\CC^1
\ee
and that the operators
\beq
\label{Fr20i}
A_i^2, D_i^2, A_iD_i, D_iA_i, A_iB_i, \mbox{ and } D_iB_i  
\mbox{ map  continuously } \CC \mbox{ into } \tilde\CC^1.
\ee

To prove  \reff{Fr20i}, we follow a similar argument 
as in the proof of the corresponding  property  \reff{Fr20} in the 
base case of the induction.
The only difference is that now in all the calculations involved
we use $c_j^i$ and $d_j^i$ in place of $c_j$ and $d_j$, respectively. In particular, 
 to prove smoothing property \reff{Fr20i} for $A_i^2$, on the 
account of the definitions of $A$ and $A_i$, we can follow the argument 
as in the proof of this property for $A^2$ and, hence, reduce the problem to  the one for the operator 
$JA_i$
 where
the operator $A_i$ is defined by \reff{Aidef}. 

It remains to prove the bijectivity property \reff{isom}. 
Again, following the same argument as in the base case, we actually have 
to show that the system
$$
v(t)=\left[B_i^\prime v\right](t)+g(t),
$$
where the operator $B_i^\prime\in\LL\left(C_T(\R)\right)$ is given by
\beq
\left[B_i^\prime v\right](t)=c_1^i(0,1,t)c_2^i(1,0,\om_1(0,1,t)) v(\om_2(1,0,\om_1(0,1,t))),\nonumber
\ee
is uniquely solvable in $C_T^1(\R)$ for any $g\in C_T^1(\R)$.
The latter is true iff
\beq\label{contr1i}
I-B_i^\prime \mbox{ is bijective from }   C_T^1(\R) \mbox{ to }   C_T^1(\R).
\ee
Obviously, \reff{contr1i} is true whenever
\beq\label{contr2}
\bigl\|B_i^\prime \bigr\|_{\LL\left(C_T^1(\R)\right)}<1.
\ee
Now we show that \reff{contr2} is a consequence of  the 
 contraction condition \reff{small1} with $l=i,i+1$.
Similarly to the above we will again norm the space $C_T^1(\R)$ with \reff{beta}.
The proof is completed by showing that  there exist  constants $\ga_i<1$ and 
$\be_i<1$ such that 
\beq\label{le}
\|B_i^\prime v\|_\infty+\be_i
\left\|\frac{d}{dt}B_i^\prime v\right\|_\infty \le \ga_i\left(\|v\|_\infty+
\be_i\left\|v^\prime\right\|_\infty\right) 
\mbox{ for all } v \in C_T^1(\R).
\ee
By  assumption \reff{small1},  $\|B_i^\prime\|_{\LL\left(C_T(\R)\right)}\le  q_i<1$.
Thus, the  estimate  \reff{le} will be proved if we show that
\beq
\label{cungl}
\left\|\frac{d}{dt}B_i^\prime v\right\|_\infty \le \frac{\ga_{i}-\tilde q_i}{\be_i}
\|v\|_\infty
+\ga_ii\left\|v^\prime\right\|_\infty\quad \mbox{ for all } v \in C_T^1(\R).
\ee
Since 
\begin{eqnarray*}
\lefteqn{
\frac{d}{dt}[(B_i^\prime v)(t)]=\frac{d}{dt}\Bigl[c_1^i(0,1,t)c_2^i(1,0,\om_1(0,1,t))\Bigr] 
v(\om_2(1,0,\om_1(0,1,t)))}\\
&&+c_1^i(0,x,t)c_2^i(1,0,\om_1(0,1,t))\Bigl[\d_2\om_2(1,0,\om_1(0,1,t))\d_t\om_1(0,1,t)\Bigr]^i\\
&&\times 
v^\prime(\om_2(1,0,\om_1(0,1,t))),
\end{eqnarray*}
we get
\beq
\label{511}
\left\|\frac{d}{dt}B_i^\prime v\right\|_\infty \le 
q_i^\prime\|v\|_\infty+q_{i+1}\left\|v^\prime\right\|_\infty,
\ee
where
$
q_i^\prime
$
is given by \reff{ql_prime}.
By assumption \reff{small1} we have $q_{i}<1$ and $q_{i+1}<1$. Fix $\ga_i$ such that
$
\max\{q_i,q_{i+1}\} <\ga_i<1.
$
Then choose $\be_i$ so small that 
$$
q_i^\prime\le \frac{\ga_{i}-q_i}{\be_i}. 
$$
Finally, \reff{511} implies \reff{cungl}, what also finishes the proof of the
bijectivity property  of $I-B_i \in\LL\left(\tilde\CC^1\right)$.

\section{Smooth dependence on the data: proof of Theorem~\ref{thm:dep}}

Here we establish smooth dependence of solutions to \reff{1}--\reff{3} on
the coefficients of \reff{1} and \reff{3}. With this aim in Section~\ref{sec:setting}
we introduced a small
parameter  $\eps\ge 0$ responsible for small perturbations of the coefficients.
We therefore consider the perturbed problems \reff{1eps}--\reff{3eps} 
and  \reff{1s_eps}--\reff{3s_eps}.

 In what follows we will use the following notation:
\begin{eqnarray}
\label{cdef_ieps}
\label{Bdef1}
[B(\eps)u](x,t)&=&\Bigl(c_1^{\eps}(0,x,t)u_2(0,\om_1^{\eps}(0)),c_2^{\eps}(1,x,t) u_1(1,\om_2^{\eps}(1))\Bigr),\\
\label{Adef_eps}
[A(\eps)u](x,t)&=&\bigg(
2c_1^{\eps}(0,x,t)a^{\eps}(0,\om_1^{\eps}(0))r_0^{\eps}(\om_1^{\eps}(0))[G(\eps)u](\om_1^{\eps}(0)),\nonumber\\
&&-2c_2^{\eps}(1,x,t)a^{\eps}(1,\om_2^{\eps}(1))r_1(\om_2^{\eps}(1))[F(\eps)u](1,\om_2^{\eps}(1))
\bigg),\\
\label{Ddef_eps}
[D(\eps)u](x,t)&=&\Bigg(
-\int_0^x d_1^{\eps}(\xi,x,t) \left(b_{12}^{\eps}u_2-[a_3^{\eps}F(\eps)u]\right)(\xi,\om_1^{\eps}(\xi))\,d\xi, \nonumber\\
&&
-\int_1^x d_2^{\eps}(\xi,x,t) \left(b_{21}^{\eps}u_1-[a_3^{\eps}F(\eps)u]\right)(\xi,\om_2^{\eps}(\xi))\,d\xi
\Bigg),\\
\label{G_eps}
[G(\eps)u](t)&=& [Iu](t)+N(\eps)u,\\
\label{F_eps}
[F(\eps)u](x,t)&=& [G(\eps)u](t)+[J(\eps)u](x,t),
\\
\label{J_eps}
[J(\eps)u](x,t)&=& \int_0^x\frac{u_1(\xi,t)-u_2(\xi,t)}{2a^{\eps}(\xi,t)}\,d\xi,\\
\label{N_eps}
N(\eps)u&=& \frac{1}{C}\int_0^T\left(\frac{u_1(0,t)-u_2(0,t)}{2}-a^{\eps}(0,t)r_0^{\eps}(t)[Iu](t)\right)\,dt,
\\
c_j^{i\eps}(\xi,x,t)&=&\exp \int_x^\xi(-1)^j
\left(\frac{b_{jj}^{\eps}}{a^{\eps}}-i\frac{\d_ta^{\eps}}{a^{\eps 2}}\right)(\eta,\om_j^{\eps}(\eta,x,t))\,d\eta,\\
\label{ddef_ieps}
d_j^{i\eps}(\xi,x,t)&=&\frac{(-1)^jc_j^{i\eps}(\xi,x,t)}{a^{\eps}(\xi,\om_j^{\eps}(\xi,x,t))},
\\
\label{Bdef_eps}
[B_i(\eps)u](x,t)&=&\Bigl(c_1^{i\eps}(0,x,t)u_2(0,\om_1^{\eps}(0)),c_2^{i\eps}(1,x,t) u_1(1,\om_2^{\eps}(1))\Bigr),\\
\label{Aidef_eps}
[A_i(\eps)v](x,t)&=&\bigg(0,
-2c_2^{i\eps}(1,x,t)r_1^{\eps}(\om_2(1))[a^{\eps}J(\eps)v](1,\om_2^{\eps}(1))
\bigg),\\
\label{Didef_eps}
[D_i(\eps)u](x,t)&=&\Bigg(
-\int_0^x d_1^{i\eps}(\xi,x,t) \left(b_{12}^{\eps}v_2-[a_3^{\eps}J(\eps)v]\right)(\xi,\om_1^{\eps}(\xi))\,d\xi, \nonumber\\
&&
-\int_1^x d_2^{i\eps}(\xi,x,t) \left(b_{21}^{\eps}v_1-[a_3^{\eps}J(\eps)v]\right)(\xi,\om_2^{\eps}(\xi))\,d\xi
\Bigg),
\end{eqnarray}
where $i=1,2,\dots$ and $\om_j^{\eps}(\xi,x,t)$ is the solution to the initial value 
problem \reff{char} with $a^\eps$ in
place of $a$.

From the definitions of $c_j^\eps$ and $\om_j^\eps$ and  the regularity assumption \reff{k-smooth_dep}
one can easily derive the bounds (needed to prove Lemma~\ref{cl:iso} below)
\beq\label{0_eps}
\begin{array}{cc}
\left\|c_1^{i\eps^\prime}(0,x,t)-c_1^{i\eps^{\prime\prime}}(0,x,t)\right\|_{\CC}=
O\left(\left|\eps^\prime-\eps^{\prime\prime}\right|\right),\\[3mm]
\left\|c_2^{i\eps^\prime}(1,x,t)-c_2^{i\eps^{\prime\prime}}(1,x,t)\right\|_{\CC}=
O\left(\left|\eps^\prime-\eps^{\prime\prime}\right|\right),\\[3mm]
\left\|\om_1^{\eps^\prime}(0)-\om_1^{\eps^{\prime\prime}}(0)\right\|_{\CC}=
O\left(\left|\eps^\prime-\eps^{\prime\prime}\right|\right),\quad
\left\|\om_2^{\eps^\prime}(1)-\om_2^{\eps^{\prime\prime}}(1)\right\|_{\CC}=
O\left(\left|\eps^\prime-\eps^{\prime\prime}\right|\right),
\end{array}
\ee
being true for $j=1,2$,  all $\eps^\prime,\eps^{\prime\prime}<1$ and all nonnegative integers $i\le k$.

In this section without restriction of generality we will work under the assumptions 
\reff{a}, \reff{ar},  \reff{small1}, and  \reff{k-smooth_dep}. A similar argument works
if we replace \reff{small1} by
one of the conditions \reff{small11}, \reff{small111}, and \reff{small1111}.
One can easily check that in the case of $t$-idependent $a$ the following is true:
if one of the conditions \reff{small} and \reff{small+} is fulfilled,
then one of the conditions  \reff{small1},
 \reff{small11},  \reff{small111}, and  \reff{small1111} is fulfilled as well. This fact together with
Theorem \ref{thm:Fredh}, Theorem  \ref{thm:reg} $(ii)$, and Theorem~\ref{thm:dep} $(i)$ entail 
Theorem~\ref{thm:dep} $(ii)$.

To state $(i)$, it suffices to prove  the  the
 smooth dependence result for $u^\eps$ on~$\eps$:
the value of $\eps_0$ can be chosen so small that for all $\eps\le\eps_0$ 
 there exists a unique weak solution 
to (\ref{1s_eps})--(\ref{3s_eps}) which belongs to  $C_T^{k}$, and  the map
$
\eps\in[0,\eps_0)\mapsto u^\eps\in \CC^{k-\ga-1}
$
is $C^\ga$-smooth for all non-negative integers $\ga\le k-1$.

Note that  conditions \reff{a}, \reff{ar}, and \reff{small1} are stable 
with respect to small perturbations of all functions contributing into them.
Fix $\eps_0$ so small that those conditions 
are fulfilled for all  $\eps\le\eps_0$ with $a(x,t)$, $a_1(x,t)$, $a_2(x,t)$, and
 $r_0(t)$  replaced by 
$a(x,t,\eps)$, $a_1(x,t,\eps)$, $a_2(x,t,\eps)$, and
$r_0(t,\eps)$, respectively. 
Then for all
$\eps\le\eps_0$
all conditions of Theorems  \ref{thm:Fredh} and  \ref{thm:reg} are fulfilled.
It follows that, given $\eps\le\eps_0$,
  there exists a weak solution 
to (\ref{1s_eps})--(\ref{3s_eps}) which belongs to~$C_T^{k}$. The uniqueness of the weak solution 
will follow from  the bijectivity property of the  operator
$I-C(\eps)-A(\eps)-D(\eps)$.

\begin{lem}\label{cl:iso}
{\bf (i)} There is $\eps_0>0$ such that for all $\eps\le\eps_0$ 
the operator $I-B(\eps)-A(\eps)-D(\eps)$ is bijective from $\CC^1$ to  $\CC^1$
and satisfies the  estimate
\beq\label{apr}
 \left\| \left(I-B(\eps)-A(\eps)-
D(\eps)\right)^{-1}\right\|_{\LL(\CC^1)}=O(1)
\ee
uniformly in $\eps\le\eps_0$.

{\bf (ii)} Given $i\le k-1$, there is $\eps_0>0$  such that for all $\eps\le\eps_0$ 
the operator $I-B_i(\eps)-A_i(\eps)-D_i(\eps)$ is  bijective from $\CC^1$ to  $\CC^1$
and satisfies the  estimate
\beq\label{apr1}
 \left\| \left(I-B_i(\eps)-A_i(\eps)-
D_i(\eps)\right)^{-1}\right\|_{\LL(\CC^1)}=O(1)
\ee
uniformly in $\eps\le\eps_0$.

{\bf (iii)} Given $i\le k$, there is $\eps_0>0$  such that for all $\eps\le\eps_0$ 
the operator $I-B(\eps)-A(\eps)-D(\eps)$ is  bijective from $\CC^i$ to $\CC^i$
and satisfies the  estimate
\beq\label{apr2}
 \left\| \left(I-B(\eps)-A(\eps)-
D(\eps)\right)^{-1}\right\|_{\LL(\CC^i)}=O(1)
\ee
uniformly in $\eps\le\eps_0$.
\end{lem}

\begin{proof} {\bf (i)}
Recall that within the assumptions of Theorem~\ref{thm:dep} we meet all conditions of 
Theorem \ref{thm:Fredh} 
 whenever $\eps\le\eps_0$. Hence, given $\eps\le\eps_0$, the operator of 
the problem \reff{1s_eps}--\reff{3s_eps} or, the same, the operator
$I-B(\eps)-A(\eps)-D(\eps)$ 
 is Fredholm from $\CC$ to $\CC$. 
If $\dim\ker\left(I-B(\eps)-A(\eps)-D(\eps)\right)=0$, then it is bijective from
$\CC$ to $\CC$.
From  Theorem~\ref{thm:reg} it follows that $I-B(\eps)-A(\eps)-D(\eps)$ 
 is surjective  from $\CC^1$ onto $\CC^1$. Hence, it remains to show the injectivity of 
$I-B(\eps)-A(\eps)-D(\eps)$  from $\CC$ to $\CC$, what is the same,  from $\CC^1$ onto $\CC^1$.

Suppose, on the contrary, that
there exist sequences $\eps_n\to_{n\to\infty}0$ and $u^n\in \CC^1$ such that 
\beq\label{contr0}
\|u^n\|_{\CC^1}=1
\ee
 and
\beq\label{contrr1}
u^n=B(\eps_n)u^n+A(\eps_n)u^n+D(\eps_n)u^n.
\ee
Hence,
\beq\label{contrr2}
u^n=B(\eps_n)u^n+\Bigl(A(\eps_n)+D(\eps_n)\Bigr)\Bigl(B(\eps_n)+A(\eps_n)+D(\eps_n)\Bigr)u^n.
\ee
Due to the choice of $\eps_0$, the operators $I-B(\eps)\in\LL(\CC)$ and $I-B(\eps)\in\LL(\CC^1)$ are
 invertible and satisfy the estimates
\beq\label{contr3}
\left\|(I-B(\eps))^{-1}\right\|_{\LL(\CC)}=O(1),\quad
\left\|(I-B(\eps))^{-1}\right\|_{\LL(\CC^1)}=O(1),
\ee
being uniform in $\eps\le\eps_0$. Thus, we are able to rewrite \reff{contrr2} 
as follows:
\beq\label{contr4}
\begin{array}{ccl}
u^n&=&\left(I-B(\eps_n)\right)^{-1}\Bigl(A(\eps_n)+D(\eps_n)\Bigr)\Bigl(B(\eps_n)+A(\eps_n)+
D(\eps_n)\Bigr)u^n\\
&=& \left(I-B\right)^{-1}(A+D)(B+A+D)u^n\\
&+& \Big[\left(I-B(\eps_n)\right)^{-1}\Bigl(A(\eps_n)+D(\eps_n)\Bigr)\Bigl(B(\eps_n)+A(\eps_n)+
D(\eps_n)\Bigr)\\
&-& \left(I-B\right)^{-1}(A+D)(B+A+D)\Big]u^n.
\end{array}
\ee
Let us show that the map 
\beq\label{e-cont}
\begin{array}{ll}
\eps\in[0,\eps_0)\mapsto \left(I-B(\eps)\right)^{-1}\left(A(\eps)+
D(\eps)\right)\left(B(\eps)+A(\eps)+
D(\eps)\right)\in\LL(\CC^1,\CC)\\
\mbox{ is locally Lipschitz continuous. }
\end{array}
\ee

First we show that the map 
\beq\label{map1}
\eps\in[0,\eps_0)\mapsto (I-B(\eps))^{-1}\in\LL(\CC^1,\CC)
\mbox{ is locally Lipschitz continuous. }
\ee
 Indeed, take $\eps^\prime,\eps^{\prime\prime}\le\eps_0$,
$f\in\CC^1$, and $u^\prime,u^{\prime\prime}\in\CC^1$ such that
$$
u^\prime=B(\eps^\prime) u^\prime+f,\quad u^{\prime\prime}=
B(\eps^{\prime\prime})u^{\prime\prime}+f.
$$
Hence,
$$
u^\prime-u^{\prime\prime}=\left(B(\eps^\prime)-B(\eps^{\prime\prime})\right) u^\prime
+ B(\eps^{\prime\prime})\left(u^\prime-u^{\prime\prime}\right),
$$
or, the same,
\beq\label{eq_in_C}
u^\prime-u^{\prime\prime}=\left(I-B(\eps^{\prime\prime})\right)^{-1}
\left(B(\eps^\prime)-B(\eps^{\prime\prime})\right) u^\prime.
\ee
Since $u^\prime,u^{\prime\prime}\in\CC^1$, then on the account of  \reff{Bdef_eps}, we have
\begin{eqnarray}
\lefteqn{
\left(B(\eps^\prime)-B(\eps^{\prime\prime})\right) u^\prime
=
\Biggl(
\left(c_1^{\eps^\prime}(0,x,t)-c_1^{\eps^{\prime\prime}}(0,x,t)\right)
u_2^\prime(0,\om_1^{\eps^\prime}(0))}\nonumber\\
&&+c_1^{\eps^{\prime\prime}}(0,x,t)
\int_0^1\d_2u_2^\prime\left(0,\al\om_1^{\eps^\prime}(0)+(1-\al)\om_1^{\eps^{\prime\prime}}(0)\right)\,d\al
\left(\om_1^{\eps^\prime}(0)-\om_1^{\eps^{\prime\prime}}(0)\right),\nonumber\\
&&
\left(c_2^{\eps^\prime}(1,x,t)-c_2^{\eps^{\prime\prime}}(1,x,t)\right)
u_1^\prime(1,\om_1^{\eps^\prime}(1))\label{cont_B}\\
&&+c_2^{\eps^{\prime\prime}}(1,x,t)
\int_0^1\d_2u_1^\prime\left(0,\al\om_2^{\eps^\prime}(1)+(1-\al)\om_2^{\eps^{\prime\prime}}(1)\right)\,d\al
\left(\om_2^{\eps^\prime}(1)-\om_2^{\eps^{\prime\prime}}(1)\right)
\Biggr).\nonumber
\end{eqnarray}
Thus, the equation \reff{eq_in_C} is well defined in $\CC$.
Note that, by \reff{contr3}, there is a constant $c>0$ not depending on $\eps^\prime$ and $f$ such that 
\beq\label{contr5}
\left\|u^\prime\right\|_{\CC^1}\le c\|f\|_{\CC^1}.
\ee
Now, taking into account the definition \reff{cdef_ieps} and the bounds \reff{0_eps}, \reff{contr3},  
and \reff{contr5},   from \reff{cont_B}  we derive the estimate
$$
\|u^\prime-u^{\prime\prime}\|_{\CC}\le c\left|\eps^\prime-\eps^{\prime\prime}\right|\|f\|_{\CC^1}
$$
with a new constant $c$ independent of $\eps^\prime$, $\eps^{\prime\prime}$, and $f$.
Therewith  \reff{map1}
is proved.

To finish with \reff{e-cont}, we take into account the definitions
\reff{Bdef1}--\reff{Ddef_eps} of the operators $B(\eps)$, $A(\eps)$, $D(\eps)$ and get
$B(\eps), A(\eps), D(\eps)\in\LL(\CC^1)$  as well as $B(\eps)$, $A(\eps)$, $D(\eps)\in\LL(\CC^1,\CC)$
are locally Lipschitz continuous in $\eps$. Hence, we have
$
 \left(A(\eps)+
D(\eps)\right)\left(B(\eps)+A(\eps)+
D(\eps)\right)\in\LL(\CC^1)
$ 
and also the map
$
\eps\in[0,\eps_0)\mapsto \left(A(\eps)+
D(\eps)\right)\left(B(\eps)+A(\eps)+
D(\eps)\right)\in\LL(\CC^1,\CC)
$ 
is locally Lipschitz continuous. This finishes the proof of \reff{e-cont}.

Now, returning to \reff{contr4}, we conclude that the second summand in the right-hand side 
tends to zero in $\CC$ as $n\to\infty$, 
while a subsequence of $\left(I-B\right)^{-1}(A+D)(B+A+D)u^n$
(the first summand) converges in $\CC$. 
Therefore, a subsequence of $u^n$ 
(further denoted by $u^n$ again)  
converges to a function $u\in\CC$. Our aim now is to show that
passing to the limit in \reff{contrr1} gives
\beq\label{contr6}
u=\left(B+A+D\right)u,
\ee
the equality being true in $\CC$. This means that $u\in\KK_u$, where 
$\KK_u$ the vector space of all  weak solutions to
\reff{1s}--\reff{3s} with $f=0$. Hence,
 by Theorem~\ref{thm:reg}, the function  $u$ has
 $\CC^{1}$ regularity. 
On the other hand, due to \reff{contr0},
$\|u\|_{\CC^1}=1$, a contradiction with $u\in\KK_u$  $u\in\CC^1$, and  $\dim\KK_u=0$. 

We are left with proving  \reff{contr6}. Above we showed 
 that $u^n$ and, hence, the right-hand side of  \reff{contrr1} converges in $\CC$. Thus, we are 
done if we prove that
$$
B(\eps_n)u^n\to Bu,\quad A(\eps_n)u^n\to Au,\quad D(\eps_n)u^n\to Du\quad \mbox{in } \CC
\mbox{ as } n\to\infty.
$$
Let us prove the first convergence (similar proof is true for the other two). We have
$$
B(\eps_n)u^n-Bu=(B(\eps_n)-B)u^n+B(u^n-u).
$$
The first summand in the right-hand side tends to zero in $\CC$ thanks to \reff{contr0} and the locally
Lipschitz continuity of the map 
$
\eps\in[0,\eps_0)\mapsto B(\eps)\in\LL(\CC^1,\CC)
$ 
and the second one -- due to the convergency of $u^n$ in $\CC$. 
In this way we reach the desired convergence.
Finally, passing to the limit as $n\to\infty$
in \reff{contrr1} gives \reff{contr6}.

We therefore proved that the problem \reff{1s_eps}--\reff{3s_eps}
is uniquely solvable in $\CC^1$ for each $\eps\le\eps_0$. 
To finish the proof of the claim (i), it remains to prove  the estimate \reff{apr}. 
If this is not the case, then 
there exist sequences $\eps_n\to_{n\to\infty}~0$ and $u^n\in \CC^1$ satisfying \reff{contr0}
 and
\beq
 u^n-B(\eps_n)u^n-A(\eps_n)u^n-D(\eps_n)u^n\to 0 \mbox{ as } n\to\infty \mbox{ in } \CC^1.\nonumber
\ee
We proceed similarly to the above with \reff{contrr1} in place
of \reff{contr1}, up to getting a contradiction.

{\bf (ii)} Fix an arbitrary $i\le k-1$. 
Theorem  \ref{thm:Fredh} states the Fredholmness of $I-B_i(\eps)-A_i(\eps)-D_i(\eps)$
 from $\CC^1$ to $\CC$ for all sufficiently small~$\eps$, while
Theorem~\ref{thm:reg}  strengths this result to
 the Fredholmness of  $I-B_i(\eps)-A_i(\eps)-D_i(\eps)$  from $\CC^1$ to $\CC^1$.
This means that the desired statement will be proved whenever we show 
the injectivity of  $I-B_i(\eps)-A_i(\eps)-D_i(\eps)$ from $\CC^1$ to $\CC^1$.

 Assume, conversely, that there exist sequences $\eps_n\to 0$ and $u^n\in \CC^1$ 
fulfilling \reff{contr0} and
\beq\label{contr9}
u^n=B_i(\eps_n)u^n+A_i(\eps_n)u^n+D_i(\eps_n)u^n.
\ee
Due to the choice of $\eps_0$, the operators $I-B_i(\eps)\in\LL(\CC)$ and $I-B_i(\eps)\in\LL(\CC^1)$ 
are invertible and satisfy the estimates
\beq\label{contr7}
\left\|(I-B_i(\eps))^{-1}\right\|_{\LL(\CC)}=O(1),\quad 
\left\|(I-B_i(\eps))^{-1}\right\|_{\LL(\CC^1)}=O(1)
\ee
uniformly in $\eps\le\eps_0$. This entails, in particular, that 
 there are  constants $c>0$ and 
$\eps_0$ such that
for all $\eps\le\eps_0$ and $f\in\CC^1$ the continuously differentiable solution 
to the equation $u=B_i(\eps)u+f$ satisfies the apriory estimate
\beq\label{contr8}
\left\|u\right\|_{\CC^1}\le c\|f\|_{\CC^1}.
\ee
Moreover,  from \reff{contr9} we get
\beq\label{contr10}
\begin{array}{ccl}
u^n
&=& \left(I-B_i\right)^{-1}(A_i+D_i)(B_i+A_i+D_i)u^n\\
&+& \Big[\left(I-B_i(\eps_n)\right)^{-1}\Bigl(A_i(\eps_n)
+D_i(\eps_n)\Bigr)\Bigl(B_i(\eps_n)+A_i(\eps_n)+
D_i(\eps_n)\Bigr)\\
&-& \left(I-B_i\right)^{-1}(A_i+D_i)(B_i+A_i+D_i)\Big]u^n,
\end{array}
\ee
where $B_i=B_i(0)$, $A_i=A_i(0)$, $D_i=D_i(0)$.

Next, we need that the maps
$$
\eps\in[0,\eps_0)\mapsto \left(I-B_i(\eps_n)\right)^{-1}\in\LL(\CC^1,\CC)
$$ 
and
$$
\eps\in[0,\eps_0)\mapsto\left(A_i(\eps)+
D_i(\eps)\right)\left(B_i(\eps)+A_i(\eps)+
D_i(\eps)\right)\in\LL(\CC^1,\CC)
$$ 
are locally Lipschitz continuous. Proceeding analogously to the proof of \reff{map1},
we state that the former follows from
  the smoothness assumptions on the data as well as from
the estimates \reff{contr7}, \reff{contr8}, and \reff{0_eps}, while the latter
is a consequence of the facts that
$B_i(\eps), A_i(\eps), D_i(\eps)\in\LL(\CC^1)$  and that $B_i(\eps), A_i(\eps), D_i(\eps)\in\LL(\CC^1,\CC)$
are locally Lipschitz continuous in $\eps$. This entails also the desired property
$
 \left(A_i(\eps)+
D_i(\eps)\right)\left(B_i(\eps)+A_i(\eps)+
D_i(\eps)\right)\in\LL(\CC^1).
$

Now, accordingly to \reff{contr10},  a subsequence of $u^n$
(below denoted by $u^n$ again)  
converges in $\CC$ to a function $u\in\CC$. To get a contradiction, similarly to the above,
it remains to show that $u$ satisfies the equation
\beq\label{eqi}
u=\left(B_i+A_i+D_i\right)u
\ee
in $\CC$. We derive the latter from \reff{contr9} applying the convergency of $u^n$ to $u$
in $\CC$
as well as the locally Lipschitz continuity of the maps
$
\eps\in[0,\eps_0)\mapsto B_i(\eps)\in\LL(\CC^1,\CC),
$
$
\eps\in[0,\eps_0)\mapsto A_i(\eps)\in\LL(\CC^1,\CC),
$ 
and
$
\eps\in[0,\eps_0)\mapsto D_i(\eps)\in\LL(\CC^1,\CC).
$
Equality \reff{eqi}  means that 
$u\in\ker\left(I-B_i-A_i-D_i\right)$  and, hence,
 by Theorem~\ref{thm:reg}, the function  $u$ has
 $\CC^{1}$ regularity. 
On the other hand, due to \reff{contr0},
$\|u\|_{\CC^1}=1$, a contradiction with $u\in\ker\left(I-B_i-A_i-D_i\right)$,
  $u\in\CC^1$, and  
\beq\label{5.1}
\dim\ker\left(I-B_i-A_i-D_i\right)=0.
\ee 
Briefly speaking, the latter follows from the facts that $\KK_u=0$ and that any solution to~\reff{contr6}
has $\CC^k$-regularity, what, on the account of the proof of Theorem~\ref{thm:reg},
 necessarily leads to the unique solvability of \reff{eqi} in $\CC$
for every $i\le k$. 

To prove \reff{5.1} in details, we will use the induction on $i$. To prove the {\it base case} $i=1$,
given $f\in \CC^k$, let $u$ be the unique solution to the equation
\beq\label{5.3}
u=\left(B+A+D\right)u+Rf.
\ee
By Theorem~\ref{thm:reg}, $u\in\CC^k$. Taking into account the proof of Theorem~\ref{thm:reg},
the unique solvability of \reff{5.3} in $\CC^1$ is equivalent to the unique solvability 
in $\CC$ of the system
\beq\label{5.2}
\begin{array}{rcl}
u&=&\left(B+A+D\right)u+Rf,\\
v&=&\left(B_1+A_1+D_1\right)v+R_1u
\end{array}
\ee
with respect to $(u,v)$. Here $v=\d_tu$. As $\K_u=0$, the first equation in \reff{5.2}
is uniquely solvable in $\CC$, and we have
$$
u=\left[I-B-A-D\right]^{-1}Rf.
$$
Now it remains to note that the system \reff{5.2} is uniquely solvable in $\CC$ iff
the operator $I-B_1-A_1-D_1\in\LL(\CC)$ is bijective from $\CC$ to $\CC$. 
The base case is therewith 
proved. Given $2\le i\le k-1$, assume that $\dim\ker\left(I-B_j-A_j-D_j\right)=0$ in $\CC$ for all 
 $2\le j\le i-1$ ({\it induction assumption}) and prove that 
$\dim\ker\left(I-B_i-A_i-D_i\right)=0$ in $\CC$ ({\it induction step}).
Since the solution $u$ to  \reff{5.3} belongs to $\CC^k$, then it satisfies the equation 
\beq
w=\left(B_i+A_i+D_i\right)w+R_iu,\nonumber
\ee
where $w=\d_t^iu$ and the operator $R_i$ given by \reff{Gidef}, \reff{NQ}, and \reff{NS}
is linear operator of $u$, $\d_tu,\dots,\d_t^{i-1}u$. Note that $u$, $\d_tu,\dots,\d_t^{i-1}u$
are the continuous functions uniquely determined due to the induction assumption. Again,
because of  \reff{5.3} is uniquely solvable in $\CC^k$ and $[R_iu](x,t)$ is a known
continuous function, the equation \reff{5.3} is uniquely solvable with respect to $w$ in $\CC$
iff $I-B_i-A_i-D_i\in\LL(\CC)$ is bijective from $\CC$ to  $\CC$. The induction step is proved.

The proof of the estimate \reff{apr1} follows the same line as the proof of
\reff{apr}, what finishes the proof of this claim. 

Claim {\bf (iii)} easily follows from Claims (i) and (ii)
and  the proof of Theorem~\ref{thm:reg}.
\end{proof}

Write $B_0(\eps)=B(\eps)$,  $A_0(\eps)=A(\eps)$, and $D_0(\eps)=D(\eps)$.

\begin{lem}\label{lem:loc_Lip}
 {\bf (i)} The map
$$
\eps\in[0,\eps_0)\mapsto \left[I-B_i(\eps)-A_i(\eps)-D_i(\eps)\right]^{-1}\in\LL(\CC^1,\CC)
$$ 
is locally Lipschitz continuous for all non-negative integers $i\le k-1$. 

 {\bf (ii)} The map
$$
\eps\in[0,\eps_0)\mapsto \left[I-B(\eps)-A(\eps)-D(\eps)\right]^{-1}\in\LL(\CC^{i+1},\CC^i)
$$ 
is locally Lipschitz continuous for all non-negative integers $i\le k-1$. 
\end{lem}

Claim (i) may be proved in much the same way as the proof
 of \reff{map1} and uses now
Lemma~\ref{cl:iso} (i)--(ii). Claim (ii)  may be proved by induction on $i$ using
Lemma~\ref{cl:iso} (iii) and Lemma~\ref{lem:loc_Lip} (i).

To finish the proof of  Theorem~\ref{thm:dep}, what is left is to show that the map
$
\eps\in[0,\eps_0)\mapsto u^\eps\in \CC^{k-\ga-1}
$
is $C^\ga$-smooth for all non-negative integers $\ga\le k-1$.
The proof of this statement will be by induction on $k$.

{\bf Base case:} 
 {\it The map 
$\eps\in [0,\eps_0)\mapsto u^\eps\in \CC$ is continuous.}
The claim follows from Lemma~\ref{lem:loc_Lip} (i) with $i=0$.

Let $k\ge 2$.

{\bf Induction assumption:} {\it The map 
$\eps\in [0,\eps_0)\mapsto u^\eps\in \CC^{k-\ga-2}$ is $C^{\ga}$-smooth for all 
 non-negative integers $\ga\le k-2$.}

{\bf Induction step:} {\it The map 
$\eps\in [0,\eps_0)\mapsto u^\eps\in \CC^{k-\ga-1}$ is $C^{\ga}$-smooth for all 
 non-negative integers $\ga\le k-1$.} 
To prove the induction step, again we use induction but this time on~$\ga$.

Checking the {\it base case $\ga=0$}  we have to show that
$v^\eps=\d_t^{k-1}u^\eps$ depends continuously on~$\eps$.
Since $u^\eps\in\CC^k$, we have  $v^\eps\in\CC^{1}$. Hence $v^\eps$
 fulfills the system \reff{1s_i+1}
with $i=k-1$ pointwise where all the data involved are not fixed
at $\eps=0$ now, but are $\eps$-dependent. The latter is equivalent to the 
integral system   \reff{rep1_i}--\reff{rep2_i} with  $i=k-1$ and with 
$\eps$-dependent coefficients involved or, the same, to the operator equation
\beq\label{abstr_i1}
v^\eps=B_{k-1}(\eps)v^\eps+A_{k-1}(\eps)v^\eps+D_{k-1}(\eps)v^\eps+R_{k-1}(\eps)u^\eps,
\ee 
where
\beq
\begin{array}{lll}
[R_i(\eps)u](x,t)= \Bigg(c_1^{i\eps}(0,x,t)[S_{i}(\eps)u^\eps]_1(\om_1^\eps(0))
+\int_0^x d_1^{i\eps}(\xi,x,t) \left[Q_i(\eps)u^\eps\right]_1
(\xi,\om_1^\eps(\xi))\,d\xi, \nonumber\\
\qquad\qquad\qquad\qquad c_2^{i\eps}(1,x,t)[S_{i}(\eps)u^\eps]_2(\om_2^\eps(1))
-\int_1^x d_2^{i\eps}(\xi,x,t)\left[Q_i(\eps)u^\eps\right]_2
(\xi,\om_2^\eps(\xi))\,d\xi\Bigg),\nonumber\\
 \left[Q_i(\eps)u^\eps\right](x,t)=\Bigl(f_{1i}^\eps\left(x,t,u^\eps,\d_tu^\eps,\dots,\d_t^{i-1}u^\eps\right)-
\left[P_i(\eps)u^\eps\right](x,t),
\nonumber\\ 
\qquad\qquad\qquad\qquad f_{2i}^\eps\left(x,t,u^\eps,\d_tu^\eps,\dots,\d_t^{i-1}u^\eps\right)-\left[P_i(\eps)u^\eps\right](x,t)\Bigr),
\nonumber\\
\left[S_{i}(\eps)u^\eps\right](t)=\Biggl(2\d_t^i\left(a^\eps(0,t)r_0(t)\right)[G(\eps)u^\eps](t)+
\d_t^{i-1}\left(a^\eps(0,t)r_0^\eps(t)(u_1^\eps(0,t)+u_2^\eps(0,t))\right),\nonumber\\
\qquad\qquad\qquad\quad-
 2\d_t^i(a^\eps(1,t)r_1^\eps(t))\left[G(\eps)u^\eps\right](t)-
\d_t^{i-1}\left[a^\eps(1,t)r_1^\eps(t)(u_1^\eps(0,t)+u_2^\eps(0,t))\right]\nonumber\\
\displaystyle \qquad\qquad\qquad\quad-\d_t^{i-1}\int_0^x\d_t\left(\frac{a^\eps(1,t)r_1^\eps(t)}{a^\eps(\xi,t)}\right)
\left(u_1^\eps(\xi,t)-u_2^\eps(\xi,t)\right)\,d\xi\Biggr),\label{NS_eps}\nonumber\\
\displaystyle [P_i(\eps)u^\eps](x,t) = \d_t^ia_3^\eps\left[Iu^\eps\right](t)+\frac{1}{2}\d_t^{i-1}
\left[a_3^\eps(u_1^\eps(0,t)+u_2^\eps(0,t))\right]\nonumber\\\displaystyle \qquad\qquad\qquad\quad+
\d_t^{i-1}\int_0^x\d_t\left(\frac{a_3^\eps(x,t)}{a^\eps(\xi,t)}\right)
\frac{u_1^\eps(\xi,t)-u_2^\eps(\xi,t)}{2}\,d\xi\nonumber
\end{array}
\ee
and the operators $B_i(\eps)$,  $A_i(\eps)$,  $D_i(\eps)$, $G(\eps)\in \CC\mapsto\CC$
are defined by \reff{Bdef_eps}, \reff{Aidef_eps}, \reff{Didef_eps}, \reff{G_eps}, and \reff{F_eps}.
The continuously differentiable functions $f_{ji}^\eps$ are defined by the same rules as
$f_{ji}$ but only with $\eps$-perturbed coefficients involved.
On the account of Lemma~\ref{cl:iso} (ii) and the fact that $\left[R_{k-1}(\eps)u^\eps\right](x,t)\in\CC^1$
for each $\eps\le\eps_0$, we are able to rewrite \reff{abstr_i1} in $\CC^1$ as
$$
v^\eps=\left[I-B_{k-1}(\eps)-A_{k-1}(\eps)v^\eps-D_{k-1}(\eps)\right]^{-1}R_{k-1}(\eps)u^\eps.
$$ 
Recall that $R_{k-1}$ is a certain linear operator of $u^\eps,\d_tu^\eps,\d_t^{k-2}u^\eps$.
First we state that the map
$
\eps\in[0,\eps_0)\mapsto R_{k-1}(\eps)\in\LL(\CC^1)
$ 
is locally Lipschitz continuous, what follows from the definition of $R_{k-1}(\eps)$ 
and the regularity assumptions on the initial data.
Further we  use the induction assumption on $k$ allowing to conclude 
that  $u^\eps,\d_tu^\eps,\d_t^{k-2}u^\eps$ are locally Lipschitz continuous in $\eps$.
Finally, applying
Lemma~\ref{lem:loc_Lip} (ii) with $i=k-1$ entails the locally Lipschitz continuity of the map
$
\eps\in[0,\eps_0)\mapsto \left[I-B_{k-1}(\eps)-A_{k-1}(\eps)-D_{k-1}(\eps)\right]^{-1}\in\LL(\CC^1,\CC),
$ 
what finishes the proof of the base case $\ga=0$.

Now for an arbitrary fixed $1\le\ga\le k-1$ assume that the map 
$\eps\in[0,\eps_0)\mapsto u^\eps\in C_T^{k-\ga}$ 
is $\CC^{\ga-1}$-smooth in $\eps$ ({\it induction assumption}) and prove that 
the map 
$\eps\in[0,\eps_0)\mapsto u^\eps\in C_T^{k-\ga-1}$ is 
$C^{\ga}$-smooth  in~$\eps$ ({\it induction step}). 
Let $w^\eps$ be a classical solution to the problem
\beq\label{1s_i_ga}
\begin{array}{rcl}
\displaystyle\left(\partial_t  - a^\eps\partial_x\right)w_1^\eps &= &
 \displaystyle - b_{11}^\eps w_1^\eps - b_{12}^\eps w_2^\eps - a_3F(\eps)u^\eps
+  \tilde F(\eps)u^\eps + \d_\eps^{\ga-1}\left(\d_\eps a\d_x v_1^\eps\right)
\\
\displaystyle
\left(\partial_t  + a^\eps\partial_x\right)w_2^\eps& =&
-\displaystyle b_{21}^\eps w_1^\eps - b_{22}^\eps w_2^\eps -a_3F(\eps)u^\eps+
 \tilde F(\eps)u^\eps - \d_\eps^{\ga-1}\left(\d_\eps a\d_x v_2^\eps\right),
\end{array}
\ee
\beq\label{2s_i_ga}
w_j^\eps(x,t) = w_j^\eps(x,t+T),\quad  x \in [0,1], \; j=1,2,
\ee
\beq\label{3s_i_ga}
\begin{array}{lll}
\displaystyle
w_1^\eps(0,t) =\displaystyle w_2^\eps(0,t)+2a^\eps(0,t)r_0^\eps(t)\left[Gw^\eps\right](t)
+2\d_\eps^{\ga-1}\left[\d_\eps\left(a^\eps(0,t)r_0^\eps(t)\right)\left[Gw^\eps\right](t)\right]\\
\displaystyle
w_2^\eps(1,t) =\displaystyle w_1^\eps(1,t)-2a^\eps(1,t)r_1^\eps(t)\left[Fw^\eps\right](1,t)
-2\d_\eps^{\ga-1}\left[\d_\eps\left(a^\eps(1,t)r_1^\eps(t)\right)\left[Fw^\eps\right](1,t)\right]
\end{array}
\ee
or, the same, the problem
\beq\label{abstr_i3}
w^\eps=B(\eps)w^\eps+A(\eps)w^\eps+D(\eps)w^\eps+Q_{\ga-1}(\eps)u^\eps+\tilde R_{\ga-1}(\eps)u^\eps,
\ee 
where
\begin{eqnarray*}
\left[Q_{\ga-1}(\eps)u^\eps\right](x,t)&=&\Biggl(
\int_0^x d_1^\eps(\xi,x,t)\left[\d_\eps^{\ga-1}\left(\d_\eps a^\eps\d_xu_1^\eps\right)\right]
(\xi,\om_1^\eps(\xi))\,d\xi, 
\nonumber\\ &&
\int_1^x d_2^\eps(\xi,x,t)\left[\d_\eps^{\ga-1}\left(\d_\eps a^\eps\d_xu_2^\eps\right)\right]
(\xi,\om_2^\eps(\xi))\,d\xi
\Biggr),\nonumber\\
\left[\tilde R_{\ga-1}(\eps)u^\eps\right]_1(x,t)&=& 2c_1^\eps(0,x,t)\d_\eps^{\ga-1}\left[\d_\eps
\left(a^\eps(0,t)r_0^\eps(t)\right)\left[Gu^\eps\right](t)\right]\nonumber\\ &&+
\int_0^x d_1^\eps(\xi,x,t)\left[\tilde F_{\ga-1}(\eps)u^\eps\right](\xi,\om_1(\xi))\,d\xi,\nonumber\\ 
\left[\tilde R_{\ga-1}(\eps)u^\eps\right]_2(x,t)&=&
-2c_2^\eps(1,x,t)\d_\eps^{\ga-1}\left[\d_\eps
\left(a^\eps(1,t)r_1^\eps(t)\right)\left[Gu^\eps\right](t)\right]\nonumber\\ &&-
\int_1^x d_2^\eps(\xi,x,t)\left[\tilde F_{\ga-1}(\eps)u^\eps\right](\xi,\om_2(\xi))\,d\xi,\nonumber\\
\left[\tilde F_{\ga-1}(\eps)u\right]_1(x,t)&=&\Biggl(
-\d_\eps^{\ga-1}\left[\d_\eps b_{11}^\eps w_1^\eps + \d_\eps b_{12}^\eps w_2^\eps + a_3F(\eps)u^\eps\right],\nonumber\\
&& -\d_\eps^{\ga-1}\left[\d_\eps b_{21}^\eps w_1^\eps + \d_\eps b_{22}^\eps w_2^\eps + a_3F(\eps)u^\eps\right]
\Biggr).\nonumber
\end{eqnarray*}

First show that, given $\eps\le\eps_0$, the equation \reff{abstr_i3} is well-defined in $\CC^{k-\ga}$.
By Lemma~\ref{cl:iso}~(iii), the operator $I-B(\eps)-A(\eps)-D(\eps)$ is bijective from
 $\CC^{k-\ga}$ to $\CC^{k-\ga}$.
Hence, it remains to show that 
$\left[Q_{\ga-1}(\eps)+\tilde R_{\ga-1}(\eps)\right]u^\eps\in\CC^{k-\ga}$.
The induction assumption on $\ga$ implies that the map
\beq\label{Lip_R}
\eps\in[0,\eps_0)\mapsto \tilde R_{\ga-1}(\eps)u^\eps
\in\CC^{k-\ga} \mbox{ is locally Lipschitz continuous. }
\ee
In particular, given $\eps\le\eps_0$, $\left[\tilde R_{\ga-1}(\eps)u^\eps\right](x,t)\in\CC^{k-\ga}$ 
as desired.
In order to prove that $Q_{\ga-1}(\eps)u^\eps\in\CC^{k-\ga}$, it is sufficient to show that 
$\d_\eps^{\ga-1}\d_xu^\eps\in\CC^{k-\ga}$ or, the same, that $\d_\eps^{\ga-1}\d_tu^\eps\in\CC^{k-\ga}$.
For $\ga=1$ the statement is obvious. Using the induction argument, assume that 
$\d_\eps^{\ga-2}\d_xu^\eps\in\CC^{k-\ga+1}$  or, the same, that $\d_\eps^{\ga-2}\d_xu^\eps\in\CC^{k-\ga+1}$ 
for an arbitrary fixed $2\le\ga\le k-1$
and prove that $\d_\eps^{\ga-1}\d_xu^\eps\in\CC^{k-\ga}$. For $w^\eps=\d_\eps^{\ga-1}u^\eps$
we have the system \reff{1s_i_ga}--\reff{3s_i_ga} or the system \reff{abstr_i3}
with $\ga$ replaced by $\ga-1$. Using the induction assumption that
$\d_\eps^{\ga-2}\d_xu^\eps\in\CC^{k-\ga+1}$ and Lemma~\ref{cl:iso} (iii), we get 
$\d_\eps^{\ga-1}u^\eps\in\CC^{k-\ga+1}$. Hence,  $\d_\eps^{\ga-1}\d_xu^\eps\in\CC^{k-\ga}$
as desired. Consequently, the equation \reff{abstr_i3} determines uniquely $w^\eps$ as an element of
$\CC^{k-\ga}$.

Further we state that $w^\eps$ given by  \reff{abstr_i3}  is locally Lipschitz continuous
in $\CC^{k-\ga-1}$.
Indeed, due to the induction assumption on $\ga$ the map
$
\eps\in[0,\eps_0)\mapsto u^\eps\in\CC^{k-\ga}
$
is $\CC^{\ga-1}$-smooth in $\eps$, hence, the map 
$
\eps\in[0,\eps_0)\mapsto \d_xu^\eps\in\CC^{k-\ga-1}
$
is $\CC^{\ga-1}$-smooth in $\eps$, from what follows that the map 
$
\eps\in[0,\eps_0)\mapsto \d_\eps^{\ga-1}\d_xu^\eps\in\CC^{k-\ga-1}
$
is continuous in $\eps$. Therefore, the map 
\beq\label{Lip_Q}
\eps\in[0,\eps_0)\mapsto Q(\eps)u^\eps\in\CC^{k-\ga-1} \mbox{ is locally Lipschitz continuous. }
\ee
Combining statements \reff{Lip_R} and \reff{Lip_Q} with Lemma~\ref{lem:loc_Lip} (ii)
leads to the desired statement that $w^\eps$ 
 is locally Lipschitz continuous
in $\CC^{k-\ga-1}$.

Finally, we show that $w^\eps$
determined by  \reff{abstr_i3} is in fact $\d_\eps^{\ga}u^\eps$. To this end, we
adopt the convention that $Q_{-1}(\eps)+\tilde R_{-1}(\eps)=R(\eps)$. 
Let us consider 
 \reff{abstr_i3} with $\ga-1$ in place of $\ga$ at some $\eps\le\eps_0$ and $\eps^\prime\le\eps_0$.
We thus  have the following equalities in $\CC^{k-\ga-1}$:
\beq\label{5.5}
\begin{array}{rcl}
\d_\eps^{\ga-1}u^\eps&=&\left[B(\eps)+A(\eps)+D(\eps)\right]\d_\eps^{\ga-1}u^\eps+Q_{\ga-2}(\eps)u^\eps+
\tilde R_{\ga-2}(\eps)u^\eps,\\
\d_\eps^{\ga-1}u^{\eps^\prime}&=&\left[B(\eps^\prime)+A(\eps^\prime)+D(\eps^\prime)\right]\d_\eps^{\ga-1}
u^{\eps^\prime}+Q_{\ga-2}(\eps^\prime)u^{\eps^\prime}+
\tilde R_{\ga-2}(\eps^\prime)u^{\eps^\prime}.
\end{array}
\ee 
Using  \reff{abstr_i3} and  \reff{5.5}, we get
\begin{eqnarray*}
\lefteqn{
\left[I-B(\eps)-A(\eps)-D(\eps)\right]
\left[\d_\eps^{\ga-1}u^{\eps^\prime}-\d_\eps^{\ga-1}u^{\eps}-w^\eps(\eps^\prime-\eps)\right]
}\\
&&
=\left[I-B(\eps)-A(\eps)-D(\eps)\right]
\Bigl[\d_\eps^{\ga-1}u^{\eps^\prime}-\d_\eps^{\ga-1}u^{\eps}\\
&&
-(\eps^\prime-\eps)\left[I-B(\eps)-A(\eps)-D(\eps)\right]^{-1}
\left(Q_{\ga-1}(\eps)+
\tilde R_{\ga-1}(\eps)\right)u^\eps\Bigr]\\
&&
=\left[
\left(B(\eps)-B(\eps^\prime)\right)+\left(A(\eps)-A(\eps^\prime)\right)+
\left(D(\eps)-D(\eps^\prime)\right)
\right]\d_\eps^{\ga-1}u^{\eps^\prime}\\
&&
+\left[Q_{\ga-2}(\eps)+
\tilde R_{\ga-2}(\eps)\right]\left(u^{\eps}-u^{\eps^\prime}\right)
+(\eps^\prime-\eps)\bigg[\left(Q_{\ga-2}(\eps)-Q_{\ga-2}(\eps^\prime)\right)\\ &&+
\left(\tilde R_{\ga-2}(\eps)-\tilde R_{\ga-2}(\eps^\prime)\right)
-\left(Q_{\ga-1}(\eps)+\tilde R_{\ga-1}(\eps)\right)
\bigg]u^{\eps}
\end{eqnarray*}
The first summand is $o(|\eps^\prime-\eps|)$ in $\CC^{k-\ga-1}$, as the maps 
$\eps\in[0,\eps_0)\mapsto B(\eps)z\in\CC^{k-\ga-1}$,
$\eps\in[0,\eps_0)\mapsto A(\eps)z\in\CC^{k-\ga-1}$, and
$\eps\in[0,\eps_0)\mapsto D(\eps)z\in\CC^{k-\ga-1}$
are  locally Lipschitz continuous for all $z\in\CC^{k-\ga}$.
The second summand is $o(|\eps^\prime-\eps|)$ in $\CC^{k-\ga-1}$, since
the maps 
$\eps\in[0,\eps_0)\mapsto Q_{\ga-2}(\mu)u^\eps\in\CC^{k-\ga}$ and
$\eps\in[0,\eps_0)\mapsto \tilde R_{\ga-2}(\mu)u^\eps\in\CC^{k-\ga}$
are $C^1$-smooth for all $\mu\le\eps_0$ and $u^\eps\in\CC^k$.
The third summand is $o(|\eps^\prime-\eps|)$ in $\CC^{k-\ga-1}$, due to
the maps 
$\eps\in[0,\eps_0)\mapsto Q_{\ga-2}(\eps)z\in\CC^{k-\ga}$ and
$\eps\in[0,\eps_0)\mapsto \tilde R_{\ga-2}(\eps)z\in\CC^{k-\ga}$
are $C^1$-smooth for all $\mu\le\eps_0$ and $z\in\CC^k$ and due to the estimate
$$
\left\|Q_{\ga-1}(\eps)z\right\|_{\CC^{k-\ga-1}}+\left\|\tilde R_{\ga-1}(\eps)z\right\|_{\CC^{k-\ga-1}}=
O\left(\left\|z\right\|_{\CC^{k-\ga-1}}\right),
$$
being uniform in $\eps\le\eps_0$ and $z\in\CC^k$.

The proof of Theorem~\ref{thm:dep} is therewith complete.

\end{document}